\documentclass[11pt]{article}
\usepackage[utf8]{inputenc}
\usepackage{xcolor}
\usepackage[T1]{fontenc}

\usepackage{mathrsfs}
\usepackage{amsthm,amsmath,amsfonts,amssymb}
\usepackage{graphicx}
\usepackage{enumitem}
\setlist[itemize]{itemsep=1pt, labelindent=2em}
\setlist[enumerate]{itemsep=1pt, labelindent=2em}

\usepackage{authblk}
\usepackage{cite}
\usepackage{comment}

\let\OLDthebibliography\thebibliography
\renewcommand\thebibliography[1]{
  \OLDthebibliography{#1}
  \setlength{\parskip}{0pt}
  \setlength{\itemsep}{0pt plus 0.3ex}
}

\allowdisplaybreaks
\makeatletter
\numberwithin{equation}{section}
\numberwithin{figure}{section}
\theoremstyle{plain}
\newtheorem{thm}{\protect\theoremname}
  \theoremstyle{plain}
  \newtheorem{lemma}[thm]{\protect\lemmaname}

\numberwithin{thm}{section}

\newtheorem{corollary}[thm]{Corollary}
\newtheorem{proposition}[thm]{Proposition}

\newtheorem{definition}[thm]{Definition}

\theoremstyle{remark}
\newtheorem*{rem}{Remark}

\makeatother

\usepackage{babel}
\providecommand{\propname}{Proposition}

\providecommand{\lemmaname}{Lemma}
\providecommand{\theoremname}{Theorem}

\renewcommand{\Im}{\imag}
\renewcommand{\Re}{\real}

\newcommand{\ee}{\epsilon}

\newcommand{\vp}{\varphi}

\newcommand{\im}{\text{{\rm Im}\,}}
\newcommand{\sgn}{\text{{\rm sgn}\,}}

\newcommand{\I}{{\rm i}}

\newcommand{\OO}{{\bf \Lambda}}

\newcommand{\oo}{\l}

\let \le \leqslant
\let \leq \leqslant
\let \ge \geqslant
\let \geq \geqslant
\let \epsilon \varepsilon
\let \vp \varphi

\usepackage{hyperref}
\hypersetup{
    colorlinks=true, % make the links colored
    linkcolor=black, % color TOC links in blue
    urlcolor=black, % color URLs in black
    linktoc=all % 'all' will create links for everything in the TOC
}

\usepackage[font=normal]{caption}
\usepackage{floatrow}

\usepackage{mathtools}

\usepackage[dvipsnames]{xcolor}

\global\long\def\ii{\mathfrak{i}}
\global\long\def\ee{\mathrm{e}}

\renewcommand{\limsup}{\varlimsup}

\newcommand{\norm}[1]{\lVert #1 \rVert}

\newcommand{\m}[1]{\mathbb{#1}}

\renewcommand\Re{\operatorname{Re}}
\renewcommand\Im{\operatorname{Im}}

\def\mob{\mathrm{M\ddot{o}b}}
\def\Diff{\operatorname{Diff}}
\def\QS{\operatorname{QS}}

\def\Sym{\operatorname{Sym}}
\def\WP{\operatorname{WP}}

\def\b{\beta}
\def\g{\gamma}

\def\t{\theta}

\def\I{\Iota}
\def\l{\lambda}
\def\k{\kappa}

\def\o{\omega}
\def\O{\Omega}

\def\Chat{\hat{\m{C}}}
\def\eps{\varepsilon}

\def\dd{\mathrm{d}}

\def\1{\mathbf{1}}

\title{On the Loewner energy of a welding homeomorphism}
\author{Shuo Fan\footnote{Tsinghua University, email: shuofan.math@gmail.com}, \quad Fredrik Viklund\footnote{KTH Royal Institute of Technology, email: frejo@kth.se},\quad Yilin Wang\footnote{ETH Z\"urich, email: yilin.wang@math.ethz.ch}}
\date{}
\begin{document}
\maketitle
\vspace{-2.5em}
\begin{center}
   \emph{ Dedicated to Nick Makarov on the occasion of his 70th birthday }
\end{center}
\vspace{1em}
\begin{abstract}
To any Jordan curve one may associate a circle homeomorphism $\vp : \m S^1 \to \m S^1$ via conformal welding. Through this correspondence, the Loewner energy $I^L$, also known as the universal Liouville action, is a K\"ahler potential for the unique homogeneous K\"ahler metric on the universal Teichm\"uller space.  Despite this, explicit expressions for $I^L$ in terms of $\vp$ alone do not seem to be available in the literature.

In this paper, we obtain such formulas. For this, we introduce an operator $\OO_\vp$ defined using the Fourier coefficients of the function
\[
(z,w) \mapsto \log \left|\frac{\vp(z)-\vp(w)}{z-w}\right|, \qquad (z,w) \in \mathbb{S}^1 \times \mathbb{S}^1. 
\] 
We relate $\OO_\vp$ to the single-layer potential and composition operator, and prove an analog of the classical Grunsky inequalities for quasisymmetric $\vp$. We show moreover that $\vp$ is Weil--Petersson if and only if $\OO_\vp$ is Hilbert--Schmidt, and we express $I^L$ as several related Fredholm determinants as well as a regularized  Fredholm determinant.  We also treat Schatten classes and we obtain formulas in terms of Dirichlet integrals involving $\log \vp'$ and in terms of the composition operator induced by $\vp$. 
\end{abstract}

\tableofcontents

\section{Introduction}

\subsection{Background}
The Loewner energy $I^L$ is a nonnegative quantity associated with each Jordan curve $\gamma$ on the Riemann sphere $\Chat = \m C \cup \{\infty\}$. It was originally introduced in \cite{RW}, in the context of large deviations of Schramm--Loewner evolution curves. It is invariant under M\"obius transformations of $\Chat$, and is zero if and only if the curve is a circle. In a sense, $I^L(\gamma)$ measures the ``roundness'' of the curve in a conformally invariant manner.

The Loewner energy is infinite for rough curves. The class of Jordan curves of finite energy --- the Weil--Petersson quasicircles --- has attracted much attention in recent years. It has a striking number of equivalent characterizations and connects a diverse range of topics, including random conformal geometry \cite{carfagnini2023onsager, VW1, VW2}, K\"ahler geometry on universal Teichm\"uller space \cite{W2,TT06}, Berezin quantization of the Bott--Virasoro group \cite{AST_berezin}, minimal surfaces in hyperbolic $3$-space \cite{Bishop_WP}, renormalized volume of hyperbolic $3$-manifolds \cite{bbvw}, and planar Coulomb gases \cite{johansson_viklund,johansson2021strong}. Let us briefly recall one definition of the Loewner energy and elaborate on the viewpoint from Teichm\"uller theory.  

After applying a M\"obius transformation, we can assume that $\gamma$ separates $0$ and $\infty$, and we write $D$ and $D^*$ for the bounded and unbounded components of $\mathbb{C} \smallsetminus \gamma$, respectively, and $\m D = \{z \in \m C \colon |z| < 1\}$ and $\m D^*  = \{z \in \m C \colon |z| >  1\}$.  
Let $f: \mathbb{D} \to D$ and $g: \mathbb{D}^* \to D^*$ be  conformal maps (biholomorphic functions) such that $f(0)=0$ and $g(\infty) = \infty$. Let $\mathcal{D}_\O(u) = \frac{1}{\pi}\int_\O |\nabla u|^2 d^2z$ be the 2-dimensional Dirichlet integral over a domain $\O$. The \emph{Loewner energy} of the Jordan curve $\gamma$ can be expressed as
\begin{align}\label{def:loewner energy dirichlet}
I^L(\gamma)=\mathcal{D}_\mathbb{D}(\log |f'|) + \mathcal{D}_\mathbb{D^*}(\log |g'|) + 4 \log |f'(0)|/|g'(\infty)|,
\end{align}
where $g'(\infty) = \lim_{z \to \infty}g'(z)$, see \cite{W2}.

The right-hand side of \eqref{def:loewner energy dirichlet} is also known as the \emph{universal Liouville action}, introduced in \cite{TT06} in the study of the Kähler geometry of universal Teichmüller space. To describe this, note that Carath\'eodory's theorem implies that the conformal maps $f$ and $g$ extend to homeomorphisms of the closures of $\mathbb{D}$ and $\mathbb{D}^*$, respectively. The \emph{welding homeomorphism} associated with $\gamma$ is the map $\varphi : \mathbb{S}^1 \to \mathbb{S}^1$ defined by \begin{align}\label{relation-welding}\vp = g^{-1} \circ f \mid_{\mathbb{S}^1},\end{align} where we write $\mathbb{S}^1$ for the unit circle. Here we started with $\gamma$ and obtained $\vp$. The \emph{conformal welding problem} is the inverse operation which starts with $\vp$ and tries to find conformal maps $f,g$ and a curve $\gamma$ such that \eqref{relation-welding} holds. Under regularity assumptions (such as requiring $\vp$ to be quasisymmetric), this can be done and gives a correspondence between curves and homeomorphisms, see Section~\ref{sect:prel}.

It is well known that $\gamma$ is a quasicircle if and only if $\varphi$ is quasisymmetric, and we denote by $\mathrm{QS} = \mathrm{QS}(\mathbb{S}^1)$ the group of quasisymmetric homeomorphisms of the circle. 
In \cite{shen13} it was shown that $\gamma$ has finite Loewner energy if and only if $\varphi \in \mathrm{WP}(\mathbb{S}^1) \subset \mathrm{QS}(\mathbb{S}^1)$, where $\mathrm{WP}(\mathbb{S}^1)$ denotes the class of \emph{Weil--Petersson (quasisymmetric) homeomorphisms}.
This class is defined by requiring that $\vp$ be absolutely continuous and that $\log |\varphi'|$ belongs to the Sobolev space $H^{1/2}(\mathbb{S}^1)$. See Sections~\ref{sect:classes_homeo} and~\ref{sect: Fourier and H12}.

 In this way, via conformal welding, these classes of Jordan curves can be identified with the corresponding classes of circle homeomorphisms. The latter carry natural group structures given by composition, and after factoring out M\"obius invariance, we obtain the homogeneous spaces:
$$T(1) = \mob (\m S^1) \backslash \QS (\m S^1), \qquad T_0 (1) = \mob (\m S^1) \backslash \WP (\m S^1).$$
The former is the universal Teichmüller space, whereas the latter is the Weil–Petersson Teichmüller space. It was proved in \cite{TT06} that $T_0(1)$ can be endowed with a unique homogeneous Kähler metric, called the Weil–Petersson metric, and that $I^L$, viewed as a functional on circle homeomorphisms, defines a Kähler potential for this metric on $T_0(1)$. However, the underlying reason for the appearance of $I^L$ (viewed as arising from random geometry) in this context remains unclear. With a slight abuse of notation, we will write $I^L (\vp)$ for the Loewner energy of $\vp$ (identified with its welded curve).

Nevertheless, to our knowledge, all expressions for the Loewner energy of $\varphi \in \WP(\m S^1)$ that appear in the literature\footnote{However, Proposition~\ref{prop:IL_X_Y-intro} is obtained by an immediate combination of results of Takhtajan and Teo, and the formula must have been known to them.} require first solving the welding problem in order to then express the energy in some way via the curve itself or the pair of conformal maps $(f,g)$. The goal of this paper is to obtain equivalent and intrinsic expressions for the Loewner energy using only the circle homeomorphism $\varphi$. Although this question has been asked in the area for some time, it was explicitly formulated in \cite[Problem~4.1]{Wang2025_TwoOptimizationLoewnerEnergy} and \cite{MesikeppYang2025_LoewnerEnergy}; these references also provide further motivation.

Finally, note that conformal welding can be viewed naturally as a correspondence between Jordan curves on the conformal boundary of hyperbolic $3$-space and positive curves on the conformal boundary of Anti-de Sitter $3$-space. From this perspective, the Loewner energy exhibits interesting symmetries when viewed as a functional on both sides of the correspondence \cite{Wang2025_TwoOptimizationLoewnerEnergy}. These relations are not well-understood and provide another motivation for the present paper.

\subsection{Main results}
Our approach is based on representing $\vp$ by a certain operator that we introduce. This operator is of independent interest, and in addition to its links to the Loewner energy, we will also study some of its basic properties and discuss its relation to more classical notions in geometric function theory.

  Throughout, we assume that circle homeomorphisms are orientation-preserving.

\subsection*{The Grunsky operator of a circle homeomorphism}
By analogy with the classical Grunsky coefficients associated with a univalent function (see Section~\ref{sect:classical_grunsky}), we will refer to the Fourier coefficients of (the real part of) the logarithmic difference ratio of $\vp$ as the \emph{Grunsky coefficients} of $\varphi$ or 
simply the \emph{welding Grunsky coefficients}.
 \begin{definition}[Grunsky coefficients of $\varphi$]
Let $\vp$ be a circle homeomorphism and suppose that $\log (\left|\vp(z) - \vp(w) \right|/ \left|z-w\right|) \in L^1(\mathbb{S}^1 \times \mathbb{S}^1)$. Define the \emph{Grunsky coefficients} of $\varphi$ (or welding Grunsky coefficients) by 
\begin{equation}\label{eq:welding-grunsky-coeff}
\widehat{\oo}_{k, \ell} = \frac{1}{(2\pi)^2} \iint_{[0,2\pi]^2}   \log \left|\frac{\vp(\ee^{\ii s}) - \vp(\ee^{\ii t})}{\ee^{\ii s} -\ee^{\ii t}} \right| \cdot  \ee^{-\ii (ks+\ell t)} \ \dd s \ \dd t, \qquad k, \ell \in \mathbb{Z}.
\end{equation}

\end{definition}
 \begin{definition}[Grunsky matrix of $\varphi$]
 For $k,\ell \in \mathbb{Z}^* \coloneqq \mathbb{Z} \smallsetminus \{0\}$, set
 \[
 \oo_{k, \ell} \coloneqq \sqrt{|k \ell|} \,  \widehat{\oo}_{k, \ell},
 \]
 where $\widehat{\oo}_{k, \ell}$ are the Grunsky coefficients of $\vp$ as in \eqref{eq:welding-grunsky-coeff}. Define the \emph{Grunsky matrix} of $\vp$ (or welding Grunsky matrix) by
 \begin{equation}\label{eq:intro_OO_block}
 \OO_\vp \coloneqq 
\begin{pmatrix}
M & N\\
\overline{N} & \overline{M}
\end{pmatrix},
 \end{equation}
where $M$ and $N$ are given by
\[
 (M_{k, \ell})_{k, \ell \ge 1} \coloneqq \left( \oo_{k, -\ell}\right)_{k, \ell \ge 1}, \qquad  (N_{k, \ell})_{k, \ell \ge 1} \coloneqq \left( \oo_{k,\ell} \right)_{k, \ell \ge 1},
\]
so that $M^* = M$ and $N^t = N$. 
\end{definition}

Unlike $\log|\vp'|$, the Grunsky coefficients of $\vp$ are well-defined under rather weak assumptions, for instance, if $\vp$ and $\vp^{-1}$ are H\"older continuous, since this implies 
$$ \log \left|\frac{\vp(z) - \vp(w)}{z-w} \right| \in L^p(\mathbb{S}^1 \times \mathbb{S}^1)\quad \text{ for all } p < \infty.$$
In particular, by Mori's theorem, the Grunsky coefficients of $\vp \in \QS(\mathbb{S}^1)$ are well-defined. %, see Lemma~\ref{lem:Holder continuity implies L2}. 
 Moreover, the ``square summation'' of the double Fourier series associated with $(\widehat \oo_{k, \ell})$ converges to $\log (\left|\vp(z) - \vp(w) \right|/\left|z-w\right|)$ in the $L^2$-sense and a.e.\ on $\mathbb{S}^1 \times \mathbb{S}^1$, see Lemma~\ref{lemma:sjolin} and Lemma~\ref{lem:Holder continuity implies L2}.

\subsection*{Basic properties of $\OO_\vp$}
We have the following basic results for $\OO_\vp$. We write $\mob(\mathbb{S}^1)$ for the M\"obius transformations that preserve $\mathbb{S}^1$.
\begin{thm}
If $\vp \in \QS(\mathbb{S}^1)$, then  $\OO_\vp$ defines a bounded self-adjoint operator on $\ell^2_{\mathbb{Z}^*}$ by ``matrix multiplication''. 
Moreover, $\OO_{\vp_1} = \OO_{\vp_2}$ if and only if there exists $\mu \in \mob(\m S^1)$ such that $\varphi_1 = \mu \circ \varphi_2$ and so
\[
\vp \mapsto \OO_\vp
\]
descends to an injective map on  $T(1)$. 
\end{thm}

This theorem
combines several statements proved in Sections~\ref{sect:grunsky-defs},~\ref{sect: grunsky ineq} and \ref{sect:proofs}. 
See also Corollary~\ref{cor:intro_OO_two_vp} which gives the general composition rule for $\OO_\vp$.

Moreover, we prove the following bounds, which are reminiscent of the classical Grunsky inequalities for quasicircles \cite{Grunsky1939, Pom_uni}.  Recall that any quasisymmetric circle homeomorphism extends to a $K$-quasiconformal map of $\m D$, for some $K \ge 1$.
\begin{proposition}[See Corollary~\ref{cor:Grunsky-ineqs-general-version}]
    \label{prop:qs_kappa_bound}
If $\vp \in \QS(\mathbb{S}^1)$  has a $K$-quasiconformal extension to $\m D$ and 
 ${\bf h} = ((h_k)_{k \ge 1}, (h_{-k})_{k \ge 1}) \in \ell^2 = \ell^2_{\mathbb{Z}^*}$, then
\begin{equation}\label{estimate:grunsky1}  
-\frac{1}{2}\left(K-1\right) \|{\bf h}\|_{\ell^2}^2 \le \langle \OO_\vp{\bf h},{\bf h} \rangle_{\ell^2} \le \frac{1}{2}\left(1-\frac1K\right)\|{\bf h}\|_{\ell^2}^2.
\end{equation}
In particular, since $\OO_\vp$ is self-adjoint,
\begin{equation}\label{estimate:grunsky2}  
\|\OO_\vp {\bf h}\|_{\ell^2} \le \frac{1}{2}\left(K-1\right)\|{\bf h}\|_{\ell^2}
\end{equation}
and 
\begin{equation} \label{eq:norm_bound_OO}
{\bf I} - \OO_\vp \ge \frac12 \left(1 + \frac1K\right)
\end{equation}
is invertible. 
\end{proposition}% 

Here, $\langle \OO_\vp {\bf h},{\bf h} \rangle_{\ell^2} = \lim_{m \to \infty} \sum_{1 \le |k|, |\ell| \le m}  \oo_{k,-\ell}h_\ell \overline{h_k}$. 
The proof of Proposition~\ref{prop:qs_kappa_bound} is given in Section~\ref{sect: grunsky ineq}. The argument is based on the single-layer potential, similar to the method of \cite{CJV}, and the Dirichlet integral.

\subsection*{Characterizations and determinant formulas}
Our first main result characterizes classes of welding homeomorphisms in terms of $\OO_\vp$. 
\begin{thm}\label{thm:main-Q-with-Schatten}
Let $\vp \in \QS(\mathbb{S}^1)$. Then, 
\begin{itemize}
    \item $\varphi \in \mob(\m S^1)$ if and only if $\OO_\vp = 0$;
    \item $\vp \in \WP_p(\mathbb{S}^1)$ if and only if $\OO_\vp$ is $p$-Schatten,  $p > 1$;
\item $\varphi$ is symmetric if and only if $\OO_\vp$ is compact.
\end{itemize} 
In particular, $\vp \in \WP(\mathbb{S}^1)$ if and only if $\OO_\vp$ is Hilbert--Schmidt.
\end{thm}

See Section~\ref{sect:classes_homeo} for the definitions of symmetric homeomorphism and the $\WP_p(\mathbb{S}^1)$ classes, along with the corresponding classes of Jordan curves; for now, we only note that $\WP_2 (\mathbb{S}^1)=\WP(\mathbb{S}^1)$, which corresponds to the Weil--Petersson quasicircles.  
%  if $\vp$ is absolutely continuous and $\log|\vp'| \in B_p(\mathbb{S}^1)$, the Besov space; in

If $\vp \in \WP(\mathbb{S}^1)$, then the Loewner energy of the welded curve can  be expressed as a regularized Fredholm determinant.
\begin{thm}\label{thm:main-det2}
     Suppose $\vp$ is symmetric. Then there exist a unitary matrix ${\bf U}$ and a Hilbert--Schmidt diagonal matrix $D$ such that
 \[
{\bf I} - \OO_\vp =  {\bf U}  \begin{pmatrix} I - D  &  0\\ 0 & I + D \end{pmatrix}^{-1} {\bf U}^*.
\]
The entries $(\delta_k)_{k \ge 1}$ of $D$ are reciprocals of the positive Fredholm eigenvalues of the welded curve $\gamma$. 
If $\vp$ can be extended to a $K$-quasiconformal map of $\mathbb{D}$, then $\delta_k \le (K-1)/(K+1)$.

Moreover, if $\vp \in \WP(\mathbb{S}^1)$, then
\[
I^L(\vp) = -12 \log \det\nolimits_2\left({\bf I}-{\bf \Lambda_\varphi}\right)^{-1} = -12 \log \det(I-D^2),
\]
where $\det_2$ is the regularized Fredholm determinant. 

\end{thm}
See Chapter~9 of \cite{Simon2005} for background on the regularized determinant $\det\nolimits_2$ and \cite{Schiffer1957, TT06} for further discussion of the Fredholm eigenvalues of the curve $\gamma$. Note that the bound on the Fredholm eigenvalues that we obtain matches that stated in \cite{Springer_acta} and attributed to Ahlfors.

\begin{rem}
    Using the composition rule for $\OO_\vp$ (Corollary~\ref{cor:OO_two_vp}), we immediately see that the Loewner energy of $\vp$ is invariant under pre- and post-composition by M\"obius transformations. We show that the diagonal matrix $D$ is unchanged when we replace $\vp$ by $\vp^{-1}$, which gives another proof of the nontrivial fact that $I^L$ is invariant under inversion (Corollary~\ref{cor:energy_inversion}). 
\end{rem}

\bigskip

We may express $I^L (\vp)$ as a (regular) Fredholm determinant involving a minor of $\OO_\vp$.
 \begin{proposition}  
 \label{prop:main-Q2}
Let $\vp \in \QS(\mathbb{S}^1)$. Consider the block matrix representation of $\OO_\vp$ as in \eqref{eq:intro_OO_block}. Then $\vp \in \WP(\mathbb{S}^1)$  if and only if $M$ is of trace class, in which case,
\[
I^L(\vp) = 12 \log \det(I-M).
\]
\end{proposition}

The proofs of Theorem~\ref{thm:main-Q-with-Schatten}, Theorem~\ref{thm:main-det2}, and Proposition~\ref{prop:main-Q2} will be completed in Section~\ref{sect:proofs}.

\begin{corollary}\label{cor:fourier_convergence}
    If $\vp \in \WP(\mathbb{S}^1)$, then the rectangular double Fourier series converges a.e.\ on $[0,2\pi]^2$, that is, \[\lim_{m, n \to \infty} \sum_{|k| \le m}\sum_{|\ell|\le n} \widehat{\oo}_{k, \ell} \ee^{\ii (ks+\ell t)} = \log \left| \frac{\vp(\ee^{\ii s}) - \vp(\ee^{\ii t})}{\ee^{\ii s} - \ee^{\ii t}}\right|.\]  
\end{corollary}
\begin{proof}
    By Theorem~\ref{thm:main-Q-with-Schatten}, $\OO_\vp$ is Hilbert--Schmidt, which implies, \[ \sum_{k=-\infty}^\infty \sum_{\ell=-\infty}^\infty(1+|k\ell|)|\widehat{\oo}_{k, \ell}|^2 <\infty.\] The claim now follows from Sj\"olin's condition \eqref{eq:Condition 1} in Lemma~\ref{lemma:sjolin}.
\end{proof}

\subsection*{The composition operator}
We will relate $\OO_\vp$ to the normalized composition operator associated with $\vp$.

Let us briefly recall the definition of composition operators following \cite{TT06} and \cite{Nag_Sullivan}. Let $\vp \in \QS(\mathbb{S}^1)$, we consider two composition operators defined by
\[
 {\bf U}_\varphi: H^{1/2} \to H^{1/2},  \quad  u \mapsto u \circ \varphi
\]
and the normalized composition operator
\[
{\bf C}_{\vp}: \dot  
H^{1/2} \to \dot{H}^{1/2},  \quad u \mapsto  u\circ \vp - \frac{1}{2\pi}\int_0^{2\pi} %u \circ \vp  \, dt,
u \circ \vp (\ee^{\ii t}) \, \dd t
\]
where $H^{1/2} = H^{1/2} (\mathbb{S}^1, \m C)$ is the complex-valued Sobolev space $W^{1/2,2}$ on the circle, and $\dot{H}^{1/2} = H^{1/2}/\m C$, which can be identified with the subspace of $H^{1/2}$ where the zeroth Fourier coefficient vanishes.
(See Section~\ref{sect: Fourier and H12}.)
It is well known \cite{Nag_Sullivan} that these are bounded operators, but see Lemma~\ref{lem:quasiinvariance} for a proof. 

Following \cite{TT06}, we consider the canonical $\m C$-basis for $\dot H^{1/2}$, 
 \begin{equation}\label{eq:basis}
 \left\{e_k (\ee^{\ii t})\coloneqq\ee^{\ii k t}/\sqrt{k}\right\}_{k \ge 1} \cup \left\{f_k (\ee^{\ii t})\coloneqq\ee^{-\ii k t}/\sqrt{k}\right\}_{k \ge 1}.
 \end{equation}
For $k, \ell \ge 1$, let
\[
x_{k, \ell} =  \sqrt{\frac{k}{\ell}} \frac{1}{2\pi} \int_0^{2\pi} \varphi(\ee^{\ii t })^\ell \ee^{-\ii k t} \dd t,\]
\[
 y_{k, \ell} =  \sqrt{\frac{k}{\ell}} \frac{1}{2\pi} \int_0^{2\pi} \varphi(\ee^{\ii t})^{-\ell} \ee^{-\ii k t} \dd t,
\] 
and set \[X=\left(x_{k, \ell}\right)_{k,\ell \ge 1}, \qquad Y = \left(y_{k, \ell}\right)_{k,\ell \ge 1}.\] 
Then, ${\bf C}_{\vp}$ and ${\bf C}_{\vp^{-1}}$ are written as block matrices with respect to the canonical basis $((e_k)_{k\ge 1}, (f_k)_{k \ge 1})$, namely,
\begin{align}\label{eq:composition-matrix1}
{\bf C}_{\vp}=\begin{pmatrix}
X & Y\\
\overline{Y} & \overline{X}
\end{pmatrix}, \quad {\bf C}_{\vp^{-1}}=\begin{pmatrix}
X^* & -Y^t\\
-Y^* & X^t
\end{pmatrix}.
\end{align}
 In \cite{TT06} it was shown that if $\vp \in \QS$, then $X, X^*$ are invertible using the link to the classical Grunsky operators, Theorem~\ref{thm:relation between Grunsky and composition}. We will give a direct and elementary proof of the invertibility in Lemma~\ref{lem:composition-grunsky-closed} under a slightly weaker hypothesis.

The following expression of the Loewner energy follows immediately from the relation between ${\bf C}_\vp$ and the classical Grunsky operators (as derived in \cite{TT06}).

 \begin{proposition}[See Proposition~\ref{prop:IL_X_Y}] 
\label{prop:IL_X_Y-intro}
Let $\vp \in \QS(\mathbb{S}^1)$. Then $\varphi \in \WP(\mathbb{S}^1)$  if and only if $Y$ is Hilbert--Schmidt, in which case, 
\[
I^L(\vp) = 12 \log \det(XX^*) = 12 \log \det(I+ YY^*).
\]
\end{proposition}

 We show that the welding Grunsky operator $\OO_\vp$ is related to the composition operator in a simple way.
 
\begin{thm}
\label{thm:2OO}
    Let $\varphi \in \QS(\mathbb{S}^1)$. Then,
           \[
                 \OO_\vp= \frac{1}{2}({\bf I}-{\bf C}_{\vp} {\bf C}_{\vp}^*). 
                \]
\end{thm}
This theorem is proved in Section~\ref{sect:proofs}.
From this, we easily obtain the composition rule for $\OO_\vp$:

\begin{corollary}[See Corollary~\ref{cor:OO_two_vp}]\label{cor:intro_OO_two_vp}
    If $\vp_1, \vp_2 \in \QS(\mathbb{S}^1)$, we have \[\OO_{\varphi_1 \circ \varphi_2} =  {\bf C}_{\vp_2}{\OO}_{\vp_1}{\bf C}_{\vp_2}^* + {\OO}_{\vp_2}.\] In particular, if $\mu \in \mob(\m S^1)$, then $\OO_{\varphi\circ \mu}=  {\bf C}_\mu \OO_\varphi  {\bf C}_\mu^*$ is a conjugation of $\OO_\vp$ by the unitary operator ${\bf C}_\mu$. 
\end{corollary}

\subsection*{Dirichlet energy}
We now ask for an expression similar to the Dirichlet integral formula \eqref{def:loewner energy dirichlet} that involves only the welding homeomorphism.  
We prove the following result. We write $u_\pm$ for the projections onto the subspaces of positive/negative modes $H_\pm \subset \dot{H}^{1/2}$, and we discuss the correspondence between the operators in $\ell^2$ and (subspaces of) $H^{1/2}$ in Section~\ref{sect: Fourier and H12} below.

\begin{thm}\label{thm:DE}
    Let $\vp \in \WP(\mathbb{S}^1)$ and write $\psi = \varphi^{-1}$. Consider the block matrix representation of ${\bf C}_\vp$ as in \eqref{eq:composition-matrix1}. 
    Then,
    \[
    I^L(\vp) =  \| (X^*)^{-1}(\log \psi')_+\|_{H_+}^2 + \|  (\overline{X})^{-1}(\log \vp')_-\|_{H_-}^2 +   4 \log |(X^{-1})_{1,1}|.
    \]
         Moreover, the right-hand side coincides with the right-hand side of \eqref{def:loewner energy dirichlet} term by term in the same order.
\end{thm}   
Suppose $\vp \in \WP(\mathbb{S}^1)$, and consider $\vp^{\circ n}$, the $n$-fold composition of $\vp$. This is again an element of $\WP(\mathbb{S}^1)$. As in \cite{MesikeppYang2025_LoewnerEnergy}, in particular Theorem~1.2, one may ask about the growth of the energy under composition.

\begin{corollary} \label{cor:growth}
Suppose $\vp \in \WP(\mathbb{S}^1)$ extends to a $K$-quasiconformal homeomorphism of $\mathbb{D}$. Then, if $\vp^{\circ n}$ denotes the $n$-fold composition of $\vp$,
     \[
    I^L(\vp^{\circ n}) \le \left(\frac{K^{n/2}-1}{K^{1/2}-1}\right)^2\left( \| \log \psi'\|_{\dot H^{1/2}}^2 +  \| \log \vp'\|_{\dot H^{1/2}}^2 \right).
    \]
 
     In particular,
    \[
    \limsup_{n\to \infty} \frac{1}{n}\log I^L(\vp^{\circ n}) \le \log K.
    \]
\end{corollary}
The theorem and the corollary are both proved in Section~\ref{sect:dirichlet}. 

\subsection{Discussion}
Several interesting classes of Jordan curves, including the Weil--Petersson curves, can be described in terms of $\log |\vp'|$; in addition to the ones discussed in this paper, see also Table~1 in \cite{Bishop2022WP}. However, (unlike $g'$ in $\mathbb{D}^*$), $\vp'$ is not always defined, even for general $\vp \in \QS$, and cannot be used to study the geometry of the welded curve. The approach of the present paper is to replace $\log|\vp'|$ by an operator encoding the log difference ratio $\log (\left|\vp(z)-\vp(w)\right| / \left| z- w \right|)$, the latter which is defined a.e.\ in $\mathbb{S}^1 \times \mathbb{S}^1$ for any homeomorphism. The matrix $\OO_\vp$ is also well-defined under very mild conditions, e.g., assuming that $\vp, \vp^{-1}$ are H\"older continuous.

The remarkable fact is that natural spectral properties of the operator $\OO_\vp$ are in exact correspondence with natural classes of homeomorphisms and welded curves. Moreover, in the case we are primarily interested in here, the correspondence is quantitative in the sense that $I^L$ can be expressed cleanly as a determinant involving $\OO_\vp$. This interesting phenomenon can also be seen in the context of the classical Grunsky operator \cite{TT06,HS12,Jones1999Grunsky}, and the arc-Grunsky operator of \cite{CJV}.

By Theorem~\ref{thm:main-Q-with-Schatten} and Theorem~\ref{thm:main-det2}, we see that one possible way to generalize the definition of Loewner energy to Jordan curves outside the Weil--Petersson class is to consider the appropriate regularized Fredholm determinant $-12\log \det \nolimits_n ({\bf I} - \OO_\vp)^{-1}$. This works for all $p$-Weil--Petersson curves when $n \ge p$, since the corresponding $\OO_\vp$ is in the $p$-Schatten class. Moreover, this gives a conformally invariant quantity, as can be seen from the composition rule for $\OO_\vp$ and the invariance of $\det \nolimits_n$ by unitary conjugations. (However, quasicircles with corners have non-compact $\OO_\vp$, which hence lie in no Schatten class. See the discussion in Section~1.2 of \cite{johansson_viklund} for more on attempting to define a renormalized Loewner energy in the presence of corners motivated by Coulomb gas asymptotics.) 

It would be very interesting to study and characterize $\OO_\vp$ for the rough random weldings appearing in Liouville quantum gravity and SLE theory. In this case, heuristically, $\log |\vp'|$ is a log-correlated field on the circle, see \cite{AJKS, Quantum_zipper} and also \cite{VW1}. Such weldings are not quasisymmetric but typically do satisfy the H\"older conditions. Moreover, we expect that  $\OO_\vp$ is unbounded whenever it is defined and $\vp \not \in \QS$. (One could try to prove a result in this direction under suitable assumptions, starting from \eqref{eq:energy-identity}.)

\subsection*{Acknowledgments} S.F. was funded by Beijing Natural Science Foundation (JQ20001) and Tsinghua scholarship for overseas graduate studies (2023076). F.V.\ was supported by the Knut and Alice Wallenberg Foundation and the G\"oran Gustafsson Foundation.  Y.W acknowledges support from the European Union (ERC, RaConTeich, 101116694)\footnote{Views and opinions expressed are however those of the author(s) only and do not necessarily reflect those of the European Union or the European Research Council Executive Agency. Neither the European Union nor the granting authority can be held responsible for them.} and the Swiss State Secretariat for Education, Research and Innovation (SERI): MB25.00004.  This work was supported by grants from the Simons Foundation International [SFI-MPS-PP-00012621-11, F.V.] and [SFI-MPS-PP-00012621-19, Y.W.]. 

We thank Leon Takhtajan for comments on an earlier version of the paper.
\section{Preliminaries}\label{sect:prel}

\subsection{Classes of circle homeomorphisms and Loewner energy}\label{sect:classes_homeo}
Let us first recall the definition of a few classes of circle homeomorphisms that will be used. We always assume that circle homeomorphisms are orientation-preserving.

    We say that a circle homeomorphism $\varphi$ is \emph{quasisymmetric} if there exists $C < \infty$ such that 
    \begin{equation}\label{def:quasi symmetric}
        \frac{1}{C} \leq \left|\frac{\varphi(\ee^{\ii  (s+ t )})-\varphi(\ee^{\ii  s })}{\varphi(\ee^{\ii  s })-\varphi(\ee^{\ii  (s- t )})}\right| \leq C
    \end{equation}
    for all $s \in \mathbb [0,2\pi]$ and $t \in (0,2\pi)$. We write $\QS = \QS(\mathbb{S}^1)$ for the group of quasisymmetric circle homeomorphisms. This class corresponds to quasicircles via conformal welding (see Definition~\ref{df:conformal_welding}).
    
    A homeomorphism $\vp \in \QS(\mathbb{S}^1)$ if and only if it can be extended to a $K$-quasiconformal homeomorphism %$f$
    $\omega$
    of $\m D$ for some $K \ge 1$, 
    that is, $\o$ satisfies a.e.\ in $\mathbb{D}$, 
    \begin{align}\label{f dilation0}    
\frac{|\o_z| +|\o_{\bar{z}}|}{|\o_z| - |\o_{\bar{z}}|} \le K,
\end{align}
where $\o_z = \partial_z \o$ and $\o_{\bar z} = \partial_{\bar z} \o$ are the Wirtinger derivatives of $\o$.
    The smallest such $K$ is called the distortion of $\o$. 
    We emphasize that this $K$ is not the same as the constant $C$ in \eqref{def:quasi symmetric} although they are quantitatively related.

We say that $\vp$ is \emph{symmetric} (or $\vp \in \Sym(\mathbb{S}^1)$) if $\vp \in \QS(\mathbb{S}^1)$ and
\[\lim_{t\to 0}\sup_{s \in [0,2\pi]}\left|\frac{\varphi(\ee^{\ii  (s+ t )})-\varphi(\ee^{\ii  s })}{\varphi(\ee^{\ii  s })-\varphi(\ee^{\ii  (s- t )})}\right| = 1. \]
This class corresponds to the class of asymptotically conformal quasicircles, see \cite{shen2007}.

We say that $\vp$ is \emph{Weil--Petersson} (or $\vp \in \WP(\mathbb{S}^1)$) if $\vp$ is absolutely continuous and $\log |\vp'| \in H^{1/2}$. More generally, for $p > 1$, $\vp \in \WP_p(\mathbb{S}^1)$ if $\log \vp' \in B_p(\mathbb{S}^1)$, where the Besov spaces $B_p$ are defined below. The $\WP$ class corresponds to the finite energy Weil--Petersson curves, and the $\WP_p$ classes correspond to the $p$-integrable Teichmüller spaces ($p$-Weil--Petersson curves),  see \cite{WeiMatsuzaki2023} and the references therein.

We let $\Diff^\infty(\mathbb{S}^1)$ denote the group of $C^\infty$-smooth circle diffeomorphisms, and $\mob (\mathbb{S}^1)$ the group of M\"obius transformations preserving $\mathbb{S}^1$.
Among these classes, we have the following strict inclusions:
$$\mob(\mathbb{S}^1) \subset \Diff^\infty (\mathbb{S}^1) \subset \WP_{1<p<2} (\mathbb{S}^1) \subset \WP(\mathbb{S}^1)  \subset \WP_{2<p} (\mathbb{S}^1) \subset \Sym(\mathbb{S}^1) \subset \QS (\mathbb{S}^1).$$
Also note that if $r < s$, then $\WP_r(\mathbb{S}^1) \subset \WP_s(\mathbb{S}^1)$.

We can identify (classes of) quasisymmetric circle homeomorphisms with (classes of) quasicircles via conformal welding.
\begin{definition}[Conformal welding] \label{df:conformal_welding}
    Let $\vp \in \QS(\mathbb{S}^1)$. A solution to the \emph{conformal welding problem} for $\vp$ consists of an oriented Jordan curve $\gamma$ and a pair of conformal maps $(f,g)$ such that $f: \mathbb{D} \to D$  and $g: \mathbb{D}^* \to D^*$ and such that $\vp = g^{-1} \circ f \mid_{\mathbb{S}^1}$. Here, $D$ and $D^*$ are respectively the components of $\Chat \smallsetminus \gamma$ on the left and on the right of $\g$. 
\end{definition}

A solution $(\gamma; f,g)$ is not unique: if $\mu \in \operatorname{PSL}(2, \mathbb{C})$, then $(\mu \circ \gamma; \mu \circ f,\mu \circ g)$ is also a solution. We may normalize to require $f(0) = 0, f'(0)=1$ and $g(\infty) = \infty$, and in this case we have uniqueness of the normalized solution (and this normalization also implies that $D^*$ contains $\infty$ while $D$ is the bounded connected component) when $\varphi \in \QS (\mathbb{S}^1)$.  The welded Jordan curve $\g$ is a \emph{quasicircle}.

\begin{thm}[See \cite{TT06,shen13,W2}]
    A circle homeomorphism $\varphi \in \WP (\mathbb{S}^1)$ if and only if it has finite \emph{Loewner energy} (also called the \emph{universal Liouville action}):
    \[
I^L(\varphi) \coloneqq I^L(\g) \coloneqq \mathcal{D}_\mathbb{D}(\log |f'|) + \mathcal{D}_\mathbb{D^*}(\log |g'|) + 4 \log |f'(0)|/|g'(\infty)| \in [0,\infty)
\]
where $\mathcal{D}_\O(u) = \frac{1}{\pi}\int_\O |\nabla u|^2 d^2z$ is the Dirichlet energy of $u$ on a domain $\O$, and $(\g; f,g)$ is a normalized solution to the conformal welding problem for $\varphi$.
    Moreover, $I^L$ is a K\"ahler potential on the Weil--Petersson Teichm\"uller space $ \mob(\mathbb{S}^1) \backslash \WP(\mathbb{S}^1)$.
\end{thm}

\begin{rem}
    The Loewner energy is in fact well-defined for all Jordan curves (but only finite if $\gamma$ is Weil--Petersson) and satisfies $I^L(\gamma) = I^L(\mu \circ \gamma)$. Moreover, it is independent of the orientation of the curve. Therefore, without loss of generality, we only consider curves arising from a normalized solution to the welding problem and view $I^L$ as a function of $\vp$. This also implies that $I^L (\varphi) = I^L (A \circ \varphi \circ B)$ for all $A,B \in \mob(\mathbb{S}^1)$ and that 
$I^L(\varphi) = I^L(\varphi^{-1}).$ 
These symmetries are also consequences of the expressions for $I^L$ that we prove ere (see Corollary~\ref{cor:energy_inversion} and comments after Theorem~\ref{thm:main-det2}).
\end{rem}

\subsection{Fourier series and $H^{1/2}$}\label{sect: Fourier and H12}
We will work with the Hilbert spaces
\[
H^{1/2}=H^{1/2}_\mathbb{C}(\mathbb{S}^1)=\left\{u: \mathbb{S}^1 \to \mathbb{C} : u=u_0+\sum_{k=1}^\infty \Big(u_k\frac{\ee^{\ii k\theta}}{\sqrt{|k|}}+u_{-k}\frac{\ee^{-\ii k\theta}}{\sqrt{|k|}}\Big), \quad  \sum_k|u_k|^2 < \infty \right\}
\]
and
\[
\dot H^{1/2}=\{u \in H^{1/2}: u_0=0\} \simeq H^{1/2}/\m C
\]
each equipped with the homogeneous Sobolev norm:
\[
\langle f,g \rangle_{\dot H^{1/2}} = \sum_{k\in \mathbb{Z}^*} f_k\overline{g_{k}} = \sum_{k=-\infty}^\infty |k| \widehat{f}_k\overline{\widehat{g}_{k}} , \qquad \|f\|_{\dot H^{1/2}}^2 = \langle f,f \rangle,
\]
where $\widehat{f}_k = f_k/\sqrt {|k|},\, \widehat{g}_k = g_k /\sqrt{|k|}$ are the Fourier coefficients. We will refer to the family $(\ee^{\ii k \t}/\sqrt{|k|})_{k \in \m Z^*}$ as the \emph{canonical basis} of $\dot H^{1/2}$.

More generally, we write for $s \in \m R$, $$\dot H^{s} = \Bigg\{f \colon \|f\|_{\dot H^{s}}^2 =  \sum_{k \in \m Z^*} |k|^{2s} |\widehat{f}_k|^2 < \infty, \quad \widehat{f}_0=0 \Bigg\}.$$

We will write $H^{1/2}_\mathbb{R}$ for real-valued functions in $H^{1/2}$.  Note that if $f=u+ \ii v$, with $u,v \in H^{1/2}_\mathbb{R}$, then $\|f\|^2_{\dot H^{1/2}} = \|u\|^2_{\dot H^{1/2}}+\|v\|^2_{\dot H^{1/2}}$.
Let ${\bf R}_0: H^{1/2} \to \dot H^{1/2}$ be the natural projection, that is,
\[
{\bf R}_0 u = u - \widehat{u}_0.
\]
Let $H_+$ and $H_-$ be the (closed) subspaces of $\dot H^{1/2}$ of functions with only positive and negative modes, respectively.

Let $p > 1$. We will also need the (homogeneous) Besov space 
    \[
   \dot{B}_p= \Bigg\{f \colon \|f\|_{\dot{B}_p}^2 = \sum_{k \in \mathbb{Z}^*} |k|^{p-1}|\widehat{f}_k|^p < \infty, \quad \widehat{f}_0=0 \Bigg\}.
    \]
    Note that $\dot B_2 = \dot H^{1/2}$.

The \emph{Hilbert transform} ${\bf J}:\dot H^{1/2} \to \dot H^{1/2}$ is defined by
\begin{equation}
    \label{eq:Hilbert}
u_k \mapsto -\ii \, \textrm{sgn}(k)\,u_k.
\end{equation}

In the canonical basis for $\dot H^{1/2}$ we have the following matrix representation \[ 
{\bf J}=\ii \begin{pmatrix}
-I & 0\\
0 & I
\end{pmatrix}.
\]
Note that 
\[
\Phi: (u_k)_{\k \ge 1} \mapsto \sum_{k \ge 1}u_k \frac{\ee^{\ii kt}}{\sqrt{k}}
\]
gives a linear isometry of $\ell^2(\mathbb{C})$ onto $H_+$. Hence, a bounded operator $F$ on $\ell^2$ can be transplanted to a bounded operator on $H_+$ by setting $F \mapsto  \Phi F \Phi^{-1}$. This applies in particular to $X$ as in \eqref{eq:composition-matrix1} and similarly to $Y$, giving a bounded operator $H_- \to H_+$. We will use the same notation for the transplanted operators where no confusion can arise.

Let $u = \sum_{k \in \m Z}\widehat{u}_k \ee^{\ii kt}  \in H^{1/2}_{\m R}$. We write \[P[u](z) = \frac{1}{2\pi} \int_0^{2\pi}\frac{1-|z|^2}{|z-\ee^{\ii t}|^2}u(\ee^{\ii t}) \dd t, \qquad z \in \mathbb{D}\] for the Poisson integral of $u$. Then, 
\[
\mathcal{D}_{\m D}(P[u]) = \frac{1}{\pi} \int_{\m D} |\nabla P[u]|^2 dz^2 = 4 \sum_{k=1}^\infty k
|\widehat{u}_k|^2   = 2 \norm{u}^2_{\dot H^{1/2}}. \]

 We introduce the single-layer potential,
    \[
    {\bf S}[v](\ee^{\ii \t}) = \frac{1}{\pi}\int_{0}^{2\pi} v(\ee^{\ii t}) \log|\ee^{\ii t}-\ee^{\ii \t}|^{-1} \dd t.
    \]
\begin{lemma}\label{lem: single-layer-potential}
    The single-layer potential is an isometry ${\bf S} : \dot H^{-1/2} \to \dot H^{1/2}$.
\end{lemma}

\begin{proof}
    We can rewrite the expression of ${\bf S}$ as 
    \[
    {\bf S}[v](\ee^{\ii s}) = \frac{1}{2\pi}\int_{0}^{2\pi} v(\ee^{\ii t}) K(t - s) \dd t,
    \]
    where 
    $$K(s) =   2\log |\ee^{\ii s} - 1|^{-1} =  \sum_{|k| \ge 1} \frac{\ee^{\ii k s}}{|k|}.$$
    From this, we deduce that the Fourier coefficients of $K$ are given by
    $\widehat K_k =  1/|k|$ for $k \neq 0$, and hence 
    $$ \widehat {{\bf S}[v]}_k  = \widehat v_k/|k|.$$
    This implies that 
    $$\norm{{\bf S}[v]}_{\dot H^{1/2}}^2 = \sum_{|k| \ge 1} |k| \left(\frac{\widehat v_k}{|k|}\right)^2 = \sum_{|k| \ge 1} |k|^{-1} |\widehat v_k|^2  = \norm{v}_{\dot H^{-1/2}}^2$$
    which completes the proof.
\end{proof}
\begin{lemma}\label{lem:log-integral-dirichlet}
Suppose $u, v\in C^\infty(\m S^1)$. Then,
\[
2\langle u,v \rangle_{\dot H^{1/2}}=\frac{1}{\pi^2}\iint_{[0,2\pi]^2} \dot{u}(\ee^{\ii s}) \overline{\dot{v}(\ee^{\ii t})} \log|\ee^{\ii s}-\ee^{\ii t}|^{-1} \dd s \dd t.
\]
\end{lemma}
\begin{proof}
Write $u(\ee^{\ii t}) = \sum_{k\in \m Z} \widehat{u}_k \ee^{\ii kt}$ so that $\dot{u}(\ee^{\ii t}) = \ii \sum_{k\in \m Z} k \widehat{u}_k \ee^{\ii kt}$. Then,
    \begin{align*}
    \frac{1}{(2 \pi)^2} &\iint_{[0,2\pi]^2} \dot{u}(\ee^{\ii s}) \overline{\dot{v}(\ee^{\ii t})} \log|\ee^{\ii s}-\ee^{\ii t}|^{-1} \dd s \dd t \\
    &= \frac{1}{2 (2\pi)^2}\iint \left(\sum_{ |k| \ge 1} \ii k \widehat{u}_k \ee^{\ii k s}\right)\left(\sum_{|k| \ge 1} -\ii k \overline{\widehat{v}_k} \ee^{-\ii k t}\right) \sum_{|k| \ge 1} \frac{\ee^{\ii k(t-s)}}{|k|} \dd s \dd t \\
     & = \frac{1}{2} \sum_{|k| \ge 1}|k|\widehat{u}_k \overline{\widehat{v}_k},
    \end{align*}
    as claimed.
\end{proof}

Now we turn to the double-index Fourier series. We have the following result on the convergence. See Theorem~7.2 of \cite{sjolin} and the references therein.
\begin{lemma}[See \cite{sjolin}]\label{lemma:sjolin}
    Suppose $u \in L^2(\mathbb{S}^1 \times \mathbb{S}^1)$ and for $k,\ell \in \mathbb{Z}$, let \[\widehat{u}_{k, \ell} = \frac{1}{(2\pi)^2}\iint_{[0,2\pi]^2} u(\ee^{\ii s}, \ee^{\ii t})\ee^{-\ii (ks + \ell t)} \dd s \dd t\]
    be the Fourier coefficients of $u$. Consider the partial sums,
    \[
S_{m,n} [u](\ee^{\ii s}, \ee^{\ii t}) = \sum_{|k| \le m}\sum_{|\ell|\le n}\widehat{u}_{k, \ell}\ee^{\ii (ks+\ell t)}.
\] 
Then the square partial sums,
    $\lim_{n \to \infty} S_{n,n}[u] = u$ a.e.\ in $\mathbb{S}^1 \times \mathbb{S}^1$ and in $L^2$.
    
    If in addition,
    \begin{equation}\label{eq:Condition 1}
    \sum_{k, \ell}|\widehat{u}_{k, \ell}|^2(\log \min(|k|+2, |\ell|+2))^2 < \infty,
    \end{equation}
     the rectangular partial sums, $\lim_{m,n\to \infty} S_{m,n}[u] = [u]$ a.e.\ in $\mathbb{S}^1 \times \mathbb{S}^1$.
\end{lemma}
 We note that in the case of rectangular summation, without the additional assumption \eqref{eq:Condition 1}, the convergence may fail everywhere even when $u$ is continuous.

\begin{lemma}\label{lem:Holder continuity implies L2}
   Suppose $\vp, \vp^{-1}$ are both H\"older continuous circle homeomorphisms. There exists $C$ such that for all $z,w \in \mathbb{S}^1 \times \mathbb{S}^1$,
    \begin{equation}\label{oct16:1}
       \left| \log \left|\frac{\vp(z)-\vp(w)}{z-w}\right|\right| \le C |\log|z-w||
    \end{equation}
   and consequently $\log \left|\vp(z)-\vp(w)\right|/ \left| z-w\right|  \in L^p (\m S^1 \times \m S^1)$ for all $p < \infty$. In particular, this holds if $\vp \in \QS(\mathbb{S}^1)$. 
\end{lemma}
\begin{proof}
    The estimate \eqref{oct16:1} follows immediately from the assumption that $\vp$ and $\vp^{-1}$ are both H\"older continuous. The statement about $\vp \in \QS(\mathbb{S}^1)$ follows, e.g., from Corollary~3.10.3 of \cite{astala2008elliptic}.
\end{proof}

\begin{lemma}\label{lem:convergence of grunsky expansion}
    Suppose $\vp \in \operatorname{Diff}^\infty(\mathbb{S}^1)$ and $p < \infty$. There exists a constant $C$ such that for all $k, \ell$, the Fourier coefficients $\widehat \oo_{k,\ell}$ of $\log \left|\vp(z)-\vp(w)\right|/ \left| z-w\right| $  satisfy
    \[
    |\widehat{\oo}_{k, \ell}| \le C (1+k^2 + \ell^2)^{-p}.
    \]
    In particular, the rectangular partial sums $S_{n,m}$ of the Fourier series of $\log \left|\frac{\vp(z)-\vp(w)}{z-w}\right|$ converge absolutely and uniformly on $\mathbb{S}^1 \times \mathbb{S}^1$, as $n,m \to \infty$.
\end{lemma}
\begin{proof}
The lemma follows using integration by parts.    
\end{proof}

\subsection{The classical Grunsky operators}\label{sect:classical_grunsky}
The Grunsky coefficients $a_{k,\ell}$, for $k,\ell \in \mathbb{Z}$, associated with the conformal maps $f,g$ mapping to the interior and exterior of a Jordan curve with the normalization: $f(0) = 0$, $f'(0) = 1$, $g (\infty) = \infty$, respectively, are defined as follows:
\[
\log \frac{g(z)-g(w)}{z-w} =  a_{0,0} - \sum_{k=1}^\infty \sum_{\ell=1}^\infty a_{k,\ell} z^{-k}w^{-\ell}, \quad z,w \in \mathbb{D}^*;
\]
\[
\log \frac{f(z)-f(w)}{z-w} = a_{0,0} - \sum_{k=0}^\infty \sum_{\ell=0}^\infty a_{-k, -\ell} z^k w^\ell, \quad z,w\in \mathbb{D};
\]
\[
\log \frac{g(z)-f(w)}{z}= a_{0,0}-\sum_{k=1}^\infty \sum_{\ell =0}^\infty a_{k, -\ell}z^{-k} w^\ell, \quad z, 1/w \in \mathbb{D}^*.
\]
Here, $a_{0,0} = \log g'(\infty)$. 
We extend the definition by symmetry: $a_{-\ell, k} \coloneqq a_{k, -\ell}, k \ge 1, \ell \ge 0$. Then we define the following semi-infinite matrices corresponding to the pair $(f,g)$, respectively,
\[
B_1 = \left(\sqrt{k\ell} \, a_{-k, -\ell}\right)_{k,\ell \ge 1}, \quad B_2 = \left(\sqrt{k\ell} \,a_{-k,\ell}\right)_{k, \ell \ge 1};
\]
and 

\[
 B_3 = \left(\sqrt{k\ell} \, a_{k, -\ell}\right)_{k, \ell \ge 1}, \quad  B_4 =\left(\sqrt{k\ell}\, a_{k, \ell}\right)_{k,\ell \ge 1}.  
\]
Note that $B_4$ is the classical Grunsky matrix. These matrices act on $\ell^2$ by matrix multiplication. 

Given $A=A_{k, \ell}$, we let \[(\overline{A})_{k, \ell} \coloneqq \overline{A_{k, \ell}}, \quad (A^t)_{k, \ell} \coloneqq A_{\ell, k}, \quad (A^*)_{k, \ell} \coloneqq \overline{A_{\ell, k}}.\] 

\begin{lemma}[See II.2.1 of \cite{TT06}] \label{lem:Grunsky_eq}
    We have
    \[
    B_1B_1^* + B_2B_2^* = I, \qquad B_3B_3^* + B_4B_4^* = I,
    \]
    and for $j= 1,\ldots, 4$, and any $u \in \ell_{{\mathbb{Z}_+}}^2$,
    \[\|B_ju\|_{\ell^2}^2 \le \|u\|_{\ell^2}^2.\]
\end{lemma}

The classical Grunsky inequality states that $B_4$ (or $B_1$) viewed as an operator acting on $\ell_{\mathbb{Z}_+}^2$ is a contraction, see \cite{Pom_uni, Duren}. Moreover, it is a strict contraction if and only if the induced welding homeomorphism $\vp = g^{-1}\circ f$ is quasisymmetric, where the operator norm depends only on the distortion of $\vp$ \cite{Pom_uni}.  
See \cite{TT06} for more information and references, where the matrix $B_1$ is also identified with the Kirillov--Yuriev--Nag--Sullivan period mapping of $T(1)$ \cite{KirillovYuriev1988,Nag_Sullivan}.

We now have the following equivalent description of Weil--Petersson quasicircles and the Loewner energy.

\begin{thm}[See \cite{TT06}]
    A Jordan curve $\gamma$ is a Weil--Petersson quasicircle if and only if $B_1$ (or $B_4$) is Hilbert--Schmidt as an operator on $\ell^2_{\mathbb{Z}_+}$. 
Moreover, the Loewner energy is expressed as
\begin{align}\label{def:loewner energy grunsky}
    I^L(\gamma) = -12 \log \det(I-B_1B_1^*) = -12 \log \det(I-B_4B_4^*).
\end{align}
\end{thm}

\section{The composition operator}\label{sect: composition}

In this section, we recall a few properties of the composition operator associated with a quasisymmetric circle homeomorphism. 
\subsection{Basic facts}
Recall the block matrix representation of the composition operators following \cite{TT06} and \cite{Nag_Sullivan}:
\begin{align}
{\bf C}_{\vp}=\begin{pmatrix}
X & Y\\
\overline{Y} & \overline{X}
\end{pmatrix}, \quad {\bf C}_{\vp^{-1}}=\begin{pmatrix}
X^* & -Y^t\\
-Y^* & X^t
\end{pmatrix}.
\end{align}
Note that ${\bf C}_{\vp} {\bf C}_{\vp^{-1}} = {\bf C}_{\vp^{-1}}{\bf C}_{\vp}  = {\bf I}$ implies
\begin{align}\label{XX* and I}
    XX^* - YY^* = I, \qquad X^*X - Y^*Y = I,
\end{align}
and
\begin{align}\label{eq:symplectic 2}
    XY^t = YX^t, \qquad  X^* Y = Y^t \overline X. 
\end{align}

The following well-known lemma gives an explicit bound on the norm of the composition operator. Since other bounds have appeared in the recent literature (e.g., \cite{Nag_Sullivan}), we choose to provide a proof here following \cite{Ahlfors2006}.

\begin{lemma}\label{lem:quasiinvariance}
Let $\varphi \in \QS(\mathbb{S}^1)$. If $\vp$ extends to a $K$-quasiconformal homeomorphism of $\mathbb{D}$, then for every $u \in H^{1/2}_{\mathbb{R}}$,
\begin{equation}\label{eq:NS_bound}
\frac{1}{K}\mathcal{D}_{\mathbb{D}}(P[u]) \le \mathcal{D}_{\mathbb{D}}(P[{\bf C}_\vp u]) \le K \mathcal{D}_{\mathbb{D}}(P[u]). 
\end{equation}
In particular, for every $u \in H^{1/2}_\mathbb{C}$,
\[
\frac{1}{\sqrt{K}}\|u\|_{\dot H^{1/2}} \le \|{\bf C}_\vp u\|_{\dot H^{1/2}} \le \sqrt{K} \|u\|_{\dot H^{1/2}}.
\]
\end{lemma}
\begin{proof}
We extend $\varphi$ to a quasiconformal homeomorphism of $\mathbb{D}$ denoted $\omega$. Suppose that the maximal distortion of $\omega$ is $K$, i.e., for a.e.\ $z \in \mathbb{D}$,
\begin{align}\label{f dilation}    
\frac{|\omega_z| +|\omega_{\bar{z}}|}{|\omega_z| - |\omega_{\bar{z}}|} \le K.
\end{align}

Let $v=P[u]$ be the Poisson integral.
%Note that $|\nabla v|^2 = 4|v_z|^2$. 
Using the chain rule and the fact that $|\bar{\omega}_z|  = |\omega_{\bar{z}}|$ and $|v_{z}| = |v_{\bar z}| $, 
\begin{align*}
|(v\circ \omega)_z| & = |v_z \circ \omega \omega_z  + v_{\bar{z}} \circ \omega \overline{\omega}_z| \\
 & \le( |v_z| \circ \omega )|\omega_z|  + (|v_{\bar{z}}| \circ \omega) |\overline{\omega}_z| \\
 & = (|v_z| \circ \omega)(|\omega_z| + |\omega_{\bar{z}}|).
\end{align*}
Since the Jacobian $J(\omega,z) = |\omega_z|^2-|\omega_{\bar{z}}|^2$, it follows that
\begin{align*}
    \int_\mathbb{D}|(v\circ \omega)_z|^2 \dd^2 z & \le  \int_\mathbb{D}(|v_z|^2 \circ \omega)\frac{(|\omega_z| + |\omega_{\bar{z}}|)^2}{|\omega_z|^2-|\omega_{\bar{z}}|^2} J(\omega,z)\dd^2 z \\
   &  = \int_\mathbb{D}(|v_z|^2 \circ \omega)\frac{|\omega_z| + |\omega_{\bar{z}}|}{|\omega_z|-|\omega_{\bar{z}}|} J(\omega,z)\dd ^2z \\
   & \le K \int_{\mathbb{D}}|v_z|^2 \dd^2 z,
\end{align*}
where we used \eqref{f dilation} in the last step.

Note that for any real-valued function $h$, $|\nabla h|^2 = 4|h_z|^2$.
Hence, we proved:
 \[
 \mathcal{D}_{\mathbb{D}}(P[u]\circ \omega) \le K \mathcal{D}_{\mathbb{D}}(P[u]). 
 \]
 Now, by the Dirichlet principle, since $\vp = \omega$ on $\mathbb{S}^1$, we have
  \[
 \mathcal{D}_{\mathbb{D}}(P[u \circ \vp]) \le \mathcal{D}_{\mathbb{D}}(P[u]\circ \omega). 
 \]
Combining these bounds gives the estimate on the right-hand side of \eqref{eq:NS_bound}. For the estimate on the left, we use the fact that if $\vp$ can be extended to a $K$-quasiconformal map, then so can $\vp^{-1}$. The statement about the $H^{1/2}$-norms follows immediately, since $\|u\|_{\dot H^{1/2}}^2 = \|\Re u\|_{\dot H^{1/2}}^2 + \|\Im u\|_{\dot H^{1/2}}^2$.
    \end{proof}

    \subsection{Approximation lemma}
    We will use the following lemma repeatedly.
\begin{lemma}\label{lem: qs approximation}
    Let $\vp \in \QS(\mathbb{S}^1)$ and suppose that $\vp$ can be extended to a $K$-quasiconformal map of $\mathbb{C}$ fixing $0,1$, and $\infty$. There exists a sequence $\vp_n \in \Diff^\infty(\mathbb{S}^1)$ normalized as $\vp$ with the following properties.
    \begin{enumerate}
        \item  Each $\vp_n$ is $K_n$-quasisymmetric with $K_n \le K$, $\vp_n \to \vp$ uniformly, and $K_n \to K$ as $n \to \infty$. 
\item There exists $C$ depending only on $K$ such that for all $n$, $\vp_n$ and  $\vp_n^{-1}$ are $1/K$-H\"older continuous with constant $C$. 
\item For each fixed
$u \in \dot{H}^{1/2}$,
\[
\lim_{n \to \infty}\| {\bf C}_{\vp_n} u - {\bf C}_{\vp} u \|_{\dot{H}^{1/2}} =0.
\]
\item{ For each $k, \ell$, $\lim_{n \to \infty} \widehat{\oo}_{k,\ell}(\vp_n) = \widehat{\oo}_{k,\ell}(\vp)$, where $\widehat{\oo}_{k,\ell}$ denotes the corresponding Grunsky coefficient.}
    \end{enumerate}
        \end{lemma}
\begin{proof}
   We follow the proof of Corollary~5.5.8 of \cite{astala2008elliptic} for Claim~$1$. Extend $\vp$ to a quasi-conformal homeomorphism $\omega$ of the plane, symmetric with respect to reflection in $\mathbb{S}^1$ and fixing $0,1,\infty$. Let $\mu$ be the Beltrami coefficient of $\omega$, which satisfies the reflection symmetry $\mu(z) = \overline{\mu(1/\overline{z})} z^2/\overline{z}^2$. 
    For each $\eps > 0$, let 
    \[
   \mu_\eps \coloneqq  \mu \, 1_{|z|<1-\eps}, \qquad z\in \mathbb{D},
    \]
    and define $\mu_\eps$ in $\mathbb{D}^*$ by reflection. Then $\mu_\eps$ converges pointwise to $\mu$ as $\eps \to 0$.
    For each $\eps$, let $\omega_\eps$ be the unique normalized solution to the Beltrami equation $\partial_{\bar{z}} \omega_\eps = \mu_\eps \partial_z \omega_\eps$ fixing $0,1$, and $\infty$. Note that $\omega_\eps$ fixes $\mathbb{S}^1$ and is analytic in a neighborhood of $\mathbb{S}^1$.
    Moreover, the maximal distortion of $\omega_\eps$, $K_\eps \le K$ by the uniform bound on the Beltrami coefficient. Since $K$-quasiconformal maps fixing $3$ points form a normal family (see \cite[pp170-171]{astala2008elliptic}), we may extract a subsequence $\eps_n \to 0$, as $n \to \infty$, such that $\omega_{\eps_n} \to \omega$ uniformly on compacts in $\mathbb{C}$. 
    We obtain the desired sequence by setting $\vp_n \coloneqq \omega_{\eps_n}\mid_{\mathbb{S}^1}$. 
    
     For Claim~$2$, note that each $\vp_n$ (and $\vp_n^{-1}$) is the restriction of a $K$-quasiconformal map of $\mathbb{C}$ fixing $\mathbb{D}$, so H\"older continuity follows from Corollary 3.10.3 of \cite{astala2008elliptic}.

    We now prove Claim $3$: for each $u \in \dot{H}^{1/2}$, ${\bf C}_{\vp_n}u \to {\bf C}_{\vp}u$, in $\dot{H}^{1/2}$, as $n \to \infty$. For this, note that by dominated convergence, the matrix representation converges element-wise; that is, for each $k,\ell$, we have $x_{k,\ell}(\vp_n) \to x_{k,\ell}(\vp)$. This implies that for each $m$, $\|{\bf C}_{\vp_n}u_m - {\bf C}_{\vp}u_m\|_{\dot H^{1/2}} \to 0$, where $u_m = P_m u$ is the projection on the first $2m$ coordinates. Moreover, since each $\vp_n$ is $K_n$-quasisymmetric, with $K_n \le K$, Lemma~\ref{lem:quasiinvariance} implies that the sequence ${\bf C}_{\vp_n}u$ is uniformly bounded in $\dot{H}^{1/2}$. We get
    \begin{align*}
    \|{\bf C}_{\vp_n}u - {\bf C}_{\vp}u\|_{\dot H^{1/2}} &\le \|{\bf C}_{\vp_n}(u-u_m)\|_{\dot H^{1/2}} + \|{\bf C}_{\vp_n}u_m - {\bf C}_{\vp}u_m\|_{\dot H^{1/2}} + \|{\bf C}_{\vp}(u_m-u)\|_{\dot H^{1/2}} \\
        & \le 2\sqrt{K}\|u-u_m\|_{\dot H^{1/2}} + \|{\bf C}_{\vp_n}u_m - {\bf C}_{\vp}u_m\|_{\dot H^{1/2}}.
    \end{align*}
Therefore,
\begin{align*}
   \limsup_{n \to \infty}  \|{\bf C}_{\vp_n}u - {\bf C}_{\vp}u\|_{\dot H^{1/2}} & \le 2\sqrt{K}\|u-u_m\|_{\dot H^{1/2}} + \limsup_{n \to \infty} \|{\bf C}_{\vp_n}u_m - {\bf C}_{\vp}u_m\|_{\dot H^{1/2}} \\
   & = 2\sqrt{K}\|u-u_m\|_{\dot H^{1/2}} \xrightarrow[]{m\to \infty} 0,
\end{align*}
which concludes the proof of the claim.

For Claim~$4$, we know that there exists $C$ depending only on the maximal distortion of $\vp$ such that uniformly in $n$, $\left|\log (|\vp_n(z) - \vp_n(w)|/|z-w| )\right|  \le C \left|\log|z-w| \right|$. Hence, the statement follows from the dominated convergence theorem.
\end{proof}

\subsection{The composition operator and Loewner energy}
Let $\vp$ be a circle homeomorphism, and we consider the block matrix representation of ${\bf C}_\vp$  as in \eqref{eq:composition-matrix1}.

\begin{lemma}\label{lem:composition-grunsky-closed}
 If $X$ is a closed operator on $\ell^2_+ = \ell_{\mathbb{Z}_+}^2$, then $X$ and $X^*$ are invertible. In particular, this holds if $\varphi \in \QS(\mathbb{S}^1)$. 
    \end{lemma}
\begin{proof}
           By \eqref{XX* and I} we have that
       \[
       \|Xu\|^2_{\ell^2_+} = \|u\|^2_{\ell^2_+} + \|\overline{Y} u\|^2_{\ell^2_+} \ge \|u\|^2_{\ell^2_+}  
       \]
       and
       \[
       \|X^*u\|_{\ell^2_+}^2 = \|u\|^2_{\ell^2_+} + \|Y^* u\|^2_{\ell^2_+} \ge \|u\|^2_{\ell^2_+}.
       \]
       This implies  $\textrm{ker}(X) = \textrm{ker}(X^*) =\{0\}$. In particular, $X$ is injective.

Now we use the fact that $X$ is a closed operator. If $X$ were not surjective, then there would be a non-zero $y \in X(\ell^2_+)^\perp$. This implies that $y$ also belongs to $\textrm{ker}(X^*)$, leading to a contradiction.  

If $\vp \in \QS$, then $X$ is a bounded operator by Lemma~\ref{lem:quasiinvariance}. The closed graph theorem implies that $X$ is a closed operator.
\end{proof}

We will need the following results for the proof of Theorem~\ref{thm:main-Q-with-Schatten}. Recall that an operator $O: \ell^2 \to \ell^2$ is said to belong to the $p$-Schatten class if $O$ is compact and the singular values of $O$ are in $\ell^p$. If $O$ is self-adjoint, the condition reduces to the eigenvalues being in $\ell^p$, see Chapter~I.1.4 of \cite{zhu2007operator}.
\begin{thm}
    [See \cite{HS12,Jones1999Grunsky,WeiMatsuzaki2023}]\label{thm:comp_compact_HS}
    Let  $\vp \in \QS(\mathbb{S}^1)$. We have that $Y$ is compact if and only if $\vp$ is symmetric and $Y$ is Hilbert--Schmidt if and only if $\vp$ is Weil--Petersson. More generally, $\vp \in \WP_p(\mathbb{S}^1)$ if and only if $Y$ is in the $p$-Schatten class. 
\end{thm}
\begin{proof}
    The statements about compact and Hilbert--Schmidt $Y$ are Theorems 2.1 and 2.2 of \cite{HS12}. Since the commutator of ${\bf C}_\vp$ and ${\bf J}$ satisfies
 \[
   [{\bf J},{\bf C}_\vp]=  2 \ii  \begin{pmatrix}
0 & Y\\
-\overline{Y} & 0
\end{pmatrix},
 \]
 Theorem~1.4 of \cite{Jones1999Grunsky} combined with Theorem 5.5 of \cite{WeiMatsuzaki2023} implies the claim about the $p$-Schatten class. 
\end{proof}

The next important result, due to Takhtajan and Teo, connects the classical Grunsky operators with the composition operator.
\begin{thm}[Proposition~II.5.1 of \cite{TT06}]\label{thm:relation between Grunsky and composition}
    Let $\vp \in \QS(\mathbb{S}^1)$. The classical Grunsky operators $B_j, j=1,\ldots, 4$, of the pair $(f,g)$ can be expressed using the composition operator:
    \begin{align*}
         \begin{pmatrix} B_1&  B_2\\B_3&  B_4 \end{pmatrix} = \begin{pmatrix} Y&  I\\I&  -Y^* \end{pmatrix}\begin{pmatrix} \overline X^{-1}&  0\\0&  (X^*)^{-1} \end{pmatrix}.
    \end{align*}
\end{thm}
In the next result, note that $X$ and $Y$ are defined directly in terms of $\vp$. We therefore obtain our first formula for the Loewner energy, expressed in terms of $\vp$. 
\begin{proposition} 
\label{prop:IL_X_Y}
Let $\vp \in \QS(\mathbb{S}^1)$ and consider the block matrix representation of ${\bf C}_\vp$ as in \eqref{eq:composition-matrix1}. Then $\varphi \in \WP(\mathbb{S}^1)$ if and only if $YY^*$ is of trace class, in which case the Loewner energy of the welded curve satisfies
\[
I^L(\gamma) = 12 \log \det(XX^*) = 12 \log \det(I+ YY^*)
\]
 \end{proposition}
\begin{proof}
By Theorem~\ref{thm:relation between Grunsky and composition} we have $XX^* = (B_2B_2^*)^{-1}$. On the other hand, Lemma~\ref{lem:Grunsky_eq}  shows that $B_2B_2^* = I-B_1B_1^*$. So, it follows from \eqref{def:loewner energy grunsky} that 
$$I^L(\g) = - 12 \log \det B_2  B_2^* = 12 \log \det XX^*,$$
and the second equality holds since $XX^*=I+YY^*$ by \eqref{XX* and I}. 
\end{proof}

\begin{lemma}\label{lem:singular value decomposition}
    Let $\vp \in \Sym(\mathbb{S}^1)$. Then there exist two unitary matrices $U$, $V$, and a diagonal matrix $D$ with $0 \leq D \le c$ for some $c = c(\varphi) \in [0,1)$ such that
    \[
        {\bf C}_{\vp} \begin{pmatrix} V &  0\\ 0 &  \overline{V} \end{pmatrix} =  \begin{pmatrix} U &  0\\ 0 &  \overline{U}  \end{pmatrix} \begin{pmatrix} (I-D^2)^{-1/2} &  D(I-D^2)^{-1/2}\\ D(I-D^2)^{-1/2} &  (I-D^2)^{-1/2}  \end{pmatrix}.
    \]
    Or equivalently,
    \[
    {\bf C}_{\vp} \begin{pmatrix} V &  V\\ \overline V &  - \overline V  \end{pmatrix} =  \begin{pmatrix} U &  U\\ \overline U &  - \overline U \end{pmatrix} \begin{pmatrix} (\frac{I+D}{I-D})^{1/2} &  0\\ 0 &  (\frac{I-D}{I+D})^{1/2}  \end{pmatrix}.
    \]
\end{lemma}
\begin{proof}
 In \cite{HS12} it is shown that the matrix $Y$ is 
    compact
    if and only if $\vp$ is symmetric.
    We define the matrix $S = X^{-1}Y$, which is symmetric since $S^t = Y^t (X^{-1})^t = X^{-1}Y$, as $XY^t = Y X^t$.
Since the matrix $S = X^{-1}Y$ is also compact, there exist a unitary matrix $V$ and a diagonal matrix $D \geq 0$ such that $S= V D V^t$ (see, e.g., \cite[Ch.~2, Lem.~2.12]{TT06}).
Since 
$$I = XX^*- YY^* = XX^*- XSS^*X^*= X(I- SS^*)X^*,$$ 
we have $X^*X= (I-SS^*)^{-1} = V(I-D^2)^{-1}V^*$. 
Since $X^* X$ is bounded and positive, we know that there exists $c<1$ such that $D \le c$. Define $U = XV(I-D^2)^{1/2}$, which is a unitary matrix, as
\[
    U^* U = (I-D^2)^{1/2} V^* X^*XV(I-D^2)^{1/2}=  I.
\]
It follows that 
$$Y \overline V= XS \overline V= X V D V^t \overline V = XVD = UD(I-D^2)^{-1/2}$$
 and 
 $$X V = U (I-D^2)^{-1/2}.$$
From this, we obtain 
 \[
        {\bf C}_{\vp} \begin{pmatrix} V &  0\\ 0 &  \overline{V} \end{pmatrix} = \begin{pmatrix} X &  Y\\ \overline Y &  \overline X \end{pmatrix} \begin{pmatrix} V &  0\\ 0 &  \overline{V} \end{pmatrix}  =  \begin{pmatrix} U &  0\\ 0 &  \overline{U}  \end{pmatrix} \begin{pmatrix} (I-D^2)^{-1/2} &  D(I-D^2)^{-1/2}\\ D(I-D^2)^{-1/2} &  (I-D^2)^{-1/2}  \end{pmatrix} 
    \]
 as claimed. Multiplying this equality by 
$\begin{psmallmatrix} I &  I\\ I & -I \end{psmallmatrix}$ on the right and using 
$$ \begin{pmatrix} (I-D^2)^{-1/2} &  D(I-D^2)^{-1/2}\\ D(I-D^2)^{-1/2} &  (I-D^2)^{-1/2}  \end{pmatrix} \begin{pmatrix} I &  I\\ I & -I \end{pmatrix} = \begin{pmatrix} I &  I\\ I & -I \end{pmatrix} \begin{pmatrix} (\frac{I+D}{I-D})^{1/2} &  0\\ 0 &  (\frac{I-D}{I+D})^{1/2}  \end{pmatrix}, $$
we obtain the second equality.
\end{proof}

\begin{lemma}\label{lem:singular_value}
    Let  $\vp \in \Sym(\mathbb{S}^1)$. Then, there exist two unitary matrices $U$, $V$, and a diagonal matrix $D$ with $0 \leq D < c$ for some $c = c(\varphi) \in [0,1)$ such that
        \begin{align*}
         \begin{pmatrix} B_1&  B_2\\B_3&  B_4 \end{pmatrix} \begin{pmatrix} \overline U&  0\\ 0 &  V \end{pmatrix}= \begin{pmatrix} U&  0\\0& \overline V \end{pmatrix}\begin{pmatrix} D&  (I-D^2)^{1/2}\\  (I-D^2)^{1/2} &  -D\end{pmatrix}.
        \end{align*}
        In particular, the nonzero entries of $D$ are the singular values of $B_1$, and are the inverse of Fredholm eigenvalues associated with the welded quasicircle.
\end{lemma}
\begin{rem}
The identity follows easily by combining Lemma~\ref{lem:singular value decomposition} and Theorem~\ref{thm:relation between Grunsky and composition}.
In particular, $B_1 = U D U^t$, hence $B_1 B_1^* = U D^2 U^*$. This shows that the eigenvalues of $D$ are the singular values of $B_1$.

Therefore, our matrix $D$ is the same as that defined on page 80 of \cite{TT06}, where the same identity is proved and the relation between $D$ and Fredholm eigenvalues of the welded quasicircle is pointed out. The relation between the Grunsky operator $B_1$ and the Fredholm eigenvalues was first shown in \cite{Schiffer1981} for $C^3$ curves.  
\end{rem}

 From the proof of Lemma~\ref{lem:singular value decomposition}, we also see that 
    $$X = U(I-D^2)^{-1/2}V^*.$$
    From this, we immediately obtain the following corollary of Proposition~\ref{prop:IL_X_Y}.
\begin{corollary}\label{cor:energy_D}
     If $\vp \in \WP(\mathbb{S}^1)$, then
     $$I^L(\g) = 12 \log \det (XX^*) = 12 \log \det (U (I-D^2)^{-1} U^*) = - 12 \log \det (I - D^2).$$
\end{corollary} 

We also obtain another proof of the invariance of $I^L$ under $\vp \mapsto \vp^{-1}$.
\begin{corollary}\label{cor:energy_inversion}
    If $\vp \in \operatorname{Sym}(\mathbb{S}^1)$, then the $D$-matrix associated with $\vp^{-1}$ coincides with that of $\vp$. In particular, if $\vp \in \WP(\m S^1)$, then $I^L(\vp) = I^L(\vp^{-1})$. 
\end{corollary}
\begin{proof}
    Let $\vp \in \operatorname{Sym}(\mathbb{S}^1)$. We write $\widetilde D$ for the diagonal matrix of singular values introduced in Lemma~\ref{lem:singular value decomposition}, but associated with $\vp^{-1}$. From the expression of ${\bf C}_{\vp^{-1}}$ (see \eqref{eq:composition-matrix1}), the diagonal matrix $\widetilde D$ has entries equal to the singular values of 
    $$- (X^*)^{-1} Y^t = - (Y \overline X^{-1})^t = - (B_1)^t$$
    which has the same singular values as $B_1$, and therefore the same singular values as the $D$ matrix associated with $\vp$ (by Lemma~\ref{lem:singular_value}). 
\end{proof}

\section{The Grunsky operator of a circle homeomorphism}
\subsection{Definition}\label{sect:grunsky-defs}
We assume $\vp \in \QS(\mathbb{S}^1)$ throughout this section. Recall the definition of the Grunsky matrix associated with $\vp$,
 \begin{equation}\label{def:matrix-repr-of-Q}
 \OO_\varphi \coloneqq  \begin{pmatrix}
M & N\\
\overline{N} & \overline{M}
\end{pmatrix}
\end{equation}
where $M$ and $N$ are defined by
\[
 \left(M_{k,\ell}\right)_{k, \ell \ge 1}=\left( \oo_{k, -\ell}\right)_{k, \ell \ge 1}, \qquad \left(N_{k,\ell}\right)_{k, \ell \ge 1}=\left( \oo_{k,\ell} \right)_{k, \ell \ge 1}.
\]
and $\oo_{k, \ell} = \sqrt{|k\ell|} \, \widehat{\oo}_{k, \ell}$, where $\widehat \oo_{k, \ell}$ are the Fourier coefficients of $\log(\left|\vp(z) -\vp(w) \right|\!/\!\left|z-w\right|)$ as discussed in Section~\ref{sect: Fourier and H12}.

\begin{rem}
Note that it follows directly from the definition that for all $k, \ell \neq 0$, we have 
\begin{equation} \label{eq:oo_self_adjoint}
\oo_{-k, -\ell} = \overline{\oo_{k,\ell}} \quad \text{ and } \quad  \oo_{k,\ell}=\oo_{\ell,k}. 
\end{equation} 
So we can also write $(\OO_\varphi)_{k,\ell \in \m Z^*} = (\oo_{k, -\ell})_{k,\ell \in \m Z^*}$.
We have defined 
 $\OO_\vp$ 
using Fourier coefficients with \emph{mixed} signs to ensure $\OO_\vp = \OO_\vp^*$.
\end{rem}
 \begin{rem}
    We could have defined the Grunsky coefficients for $\vp$ using Fourier coefficients of $\log(\vp(z) -\vp(w))/(z-w)$. In this case, one needs to suitably choose a single-valued continuous branch. This can be done as in Proposition~2.1 of \cite{AstalaIwaniecPrauseSaksman2015} after extending $\vp$ to a normalized quasiconformal homeomorphism of the plane. However, working with the real part is sufficient for our purposes and slightly simpler. 
\end{rem}

\bigskip

We first check how the matrix $\OO_\vp$ changes under post-composition by $\mob(\m S^1)$ by direct computation.  
\begin{lemma}
\label{lem:mob}
    If $\mu \in \mob(\m S^1)$, then 
    ${\bf C}_\mu$ is unitary, and $\OO_\mu = 0$. For all $\varphi \in \QS (\m S^1)$, we have $\OO_\varphi = \OO_{\mu \circ \varphi}$.
    Therefore, $\vp \mapsto \OO_\vp$ descends to a map on the universal Teichmüller space $T(1)$. 
\end{lemma}
We will see later (Corollary~\ref{cor:OO_two_vp}) that more generally $\OO_{\varphi_1 \circ \varphi_2} = {\bf C}_{\vp_2}{\OO}_{\vp_1}{\bf C}_{\vp_2}^* + {\OO}_{\vp_2}$, which implies, in particular, $\OO_{\varphi\circ \mu} = {\bf C}_\mu \OO_\varphi  {\bf C}_\mu^*$.

   \begin{proof}
    Since M\"obius transformations
satisfy $$\mu' (z)\mu'(w) = \left(\frac{\mu(z) - \mu (w)}{ z - w}\right)^2,$$ we get
\[
    \log \left| \frac{\mu(z) - \mu(w)}{z- w} \right| = \frac12 \log|\mu'(z)| + \frac12\log |\mu'(w)| 
    \]
     and so $\widehat{\lambda}_{k, \ell} = 0$ if $|k|, |\ell| \ge 1$ since we are computing a double integral, but each term depends only on one variable. It follows that $\OO_\mu = 0$.
     On the other hand, since $\mu$ preserves $H_+$ and $H_-$, ${\bf C}_\mu$ is of the form  \[
{\bf C}_{\mu}=\begin{pmatrix}
A & 0\\
0 & \overline{A}
\end{pmatrix} \] and by \eqref{XX* and I},
$AA^* = I$. Hence, ${\bf C}_\mu {\bf C}_\mu^* = {\bf I}$.

    The fact that $\OO_{\mu \circ \vp} = \OO_{\vp}$ now follows from
\[
 \log \left| \frac{\mu \circ \vp(z) - \mu \circ \vp(w)}{z-w}\right| = \log \left| \frac{\vp(z) -\vp(w)}{z-w}\right| + \frac12 \log |\mu'(\vp(z))| + \frac12 \log  |\mu' (\vp(w))|
 \]
 since the last two terms depend only on one variable.
   \end{proof}

\subsection{Single-layer potential and Grunsky inequalities for welding}  \label{sect: grunsky ineq}

We now relate $\OO_\vp$ to quantities from potential theory and use this link to prove an analog of the classical Grunsky inequalities for quasicircles. 

Recall that we wrote  \[
    {\bf S}[v](\ee^{\ii \t}) = \frac{1}{\pi}\int_{0}^{2\pi} v(\ee^{\ii t}) \log|\ee^{\ii t}-\ee^{\ii \t}|^{-1} \dd t.
    \]for the single-layer potential, which, by Lemma~\ref{lem: single-layer-potential}, defines an isometry $\dot{H}^{-1/2} \to \dot{H}^{1/2}$.  Therefore, the operator ${\bf T}$ defined by
\begin{align}\label{def:T-operator}
u \mapsto {\bf T}[u] \coloneqq {\bf S}[\dot{u}].
\end{align}
defines an isometry $\dot{H}^{1/2} \to \dot{H}^{1/2}$. 
    Here, given $u \in \dot{H}^{1/2}$, $\dot{u} \in \dot{H}^{-1/2}$ is the weak derivative of $u (\ee^{\ii t})$ with respect to $t$ and can be defined by
\[
\widehat{\dot{u}}_k = \ii k \widehat{u}_k, \quad k \in \mathbb{Z}.
\]
Therefore,
\[
\widehat{{\bf T}[u]}_k = \ii \, \sgn(k) \, \widehat{u}_k.
\]
We already note that ${\bf T} = {\bf J}^{-1}$. 
If $u$ is smooth, then $$\dot{u}(\ee^{\ii t}) = \partial_t u(\ee^{\ii t}) = \ii u'(\ee^{\ii t})\ee^{\ii t},$$
    where 
    \[
    u'(z) = \lim_{|w-z| \to 0, \ w\in \mathbb{S}^1} \frac{u(w) - u(z)}{w-z} .
    \]
Using Lemma~\ref{lem:quasiinvariance}, we observe the following.
\begin{lemma}    \label{lemma:Q-bounded}
The operator ${\bf T} - {\bf C}_\vp {\bf T} {\bf C}_{\vp^{-1}} \colon \dot H^{1/2} \to \dot H^{1/2}$
is bounded.
\end{lemma}
More explicitly, we can understand this operator as follows. If $\vp \in \Diff^\infty(\mathbb{S}^1)$ and $u$ is smooth, then 
by a change of variables and writing $\psi = \vp^{-1}$,
we compute
\begin{align*}
   \partial_t (u\circ \psi (\ee^{\ii t}))  & = u'( \psi(\ee^{\ii t}))\ii \ee^{\ii t}\psi'(\ee^{\ii t}) \\ & =  \dot u \circ \psi(\ee^{\ii t})\frac{\ee^{\ii t}}{\psi(\ee^{\ii t})} \psi'(\ee^{\ii t})  = \dot u \circ \psi(\ee^{\ii t})|\psi'(\ee^{\ii t})|.
\end{align*}
Therefore,
\begin{align} \label{eq:CTC}
    {\bf C}_\vp {\bf T} {\bf C}_{\vp^{-1}}[u](\ee^{\ii t}) & = \frac{1}{\pi} \int_{0}^{2\pi} \dot u \circ \psi (\ee^{\ii s}) |\psi'(\ee^{\ii s})| \log \frac{1}{|\varphi(\ee^{\ii t}) - \ee^{\ii s}|} \ \dd s   \nonumber \\
    & = \frac{1}{\pi} \int_{0}^{2\pi} \dot u (\ee^{\ii s})\log \frac{1}{|\varphi(\ee^{\ii t}) - \varphi(\ee^{\ii s})|} \ \dd s.
\end{align}
Hence, we obtain
\begin{align}  \label{eq:Q-and-single-layer}  
({\bf T}-{\bf C}_\vp {\bf T} {\bf C}_{\vp^{-1}})u(\ee^{\ii t}) = \frac{1}{\pi}\int_{0}^{2\pi}\dot{u}(\ee^{\ii s}) \log \left|\frac{\vp(\ee^{\ii s})-\vp(\ee^{\ii t})}{e^{\ii s}-e^{\ii t}}\right| \dd s.
\end{align}

We may now express \eqref{eq:Q-and-single-layer} using the $\OO_\vp$-matrix.
\begin{lemma}\label{lemma:Q-well-defined}
    Let $\vp \in \QS(\mathbb{S}^1)$ and $u\in \dot{H}^{1/2}$. Then if
    \[v = ({\bf T}-{\bf C}_\vp {\bf T} {\bf C}_\vp^{-1})u,\]
    we have, for each $k \neq 0$,
    \begin{equation}\label{eq:fourier_v}
        \frac{1}{2} \sqrt{|k|} \widehat{v}_{k} = \lim_{M\to \infty} \sum_{1 \le |\ell| \le M} \oo_{k, \ell} \frac{\widehat{\dot{u}}_{-\ell}}{\sqrt{|\ell|}} = \lim_{M\to \infty} \sum_{1 \le|\ell| \le M} \oo_{k,-\ell} \frac{\widehat{\dot{u}}_{\ell}}{\sqrt{|\ell|}}.
    \end{equation}
\end{lemma}

\begin{proof} 
Assume first that $\vp \in \Diff^\infty (\m S^1)$ and $u \in C^\infty (\m S^1)$.
  Using \eqref{eq:Q-and-single-layer} and Lemma~\ref{lem:convergence of grunsky expansion}, we have
\begin{align}\label{eq:smooth_v_k}
    \widehat{v}_k  & =  \frac{1}{2\pi} \int_{0}^{2\pi} \ee^{-\ii k t}  \left(\frac{1}{\pi} \int_0^{2\pi} \sum_{n} \widehat{\dot{u}}_{n} \ee^{\ii n s} \sum_{m, \ell}\widehat{\oo}_{m, \ell} \ee^{\ii ( m t +\ell s) } \ \dd s \right)  \dd t \nonumber \\
    & =  \frac{1}{2\pi} \int_{0}^{2\pi} \ee^{-\ii k t} \left(2 \sum_{m, \ell} \widehat{\dot{u}}_{-\ell} \widehat{\oo}_{m, \ell}  \ee^{\ii m t} \right) \dd t \nonumber \\
    & = 2 \sum_{\ell} \widehat{\dot{u}}_{-\ell}\widehat{\oo}_{k, \ell}.
\end{align}
For $\varphi \in \QS (\m S^1)$ and $u \in \dot H^{1/2}$, approximate $\varphi$ by a sequence  $\varphi_n \in \Diff^{\infty} (\m S^1)$ as in Lemma~\ref{lem: qs approximation}. In particular, the corresponding coefficients converge:
\begin{equation}\label{eq:l_n_converges}
    \widehat \oo_{k,\ell} (\varphi_n) \xrightarrow{n\to \infty} \widehat \oo_{k,\ell} (\varphi).
\end{equation}

Let $m \ge 1$, and write $u_m = P_m (u)$ where  $P_m$ is the projection onto the Fourier modes with index $|k| \le m$ and define
\[
v_{n,m} \coloneqq ({\bf T}-{\bf C}_{\vp_n} {\bf T} {\bf C}_{\vp_n^{-1}})u_m, \quad v_{m} \coloneqq ({\bf T}-{\bf C}_{\vp} {\bf T} {\bf C}_{\vp^{-1}})u_m,
\]
By \eqref{eq:l_n_converges}, we know that 
$$v_{n,m} \xrightarrow[]{n\to \infty} v_m, \quad \text{in } \dot H^{1/2}$$
as both sides have only finitely many nonzero Fourier coefficients. Therefore,
\[
\lim_{m \to \infty} \lim_{n \to \infty}\|v-v_{n,m}\|_{\dot H^{1/2}} = \lim_{m \to \infty} \|v - v_m\|_{\dot H^{1/2}} 
=0,
 \]
 where the second equality follows from the fact that ${\bf T} - {\bf C}_{\vp} {\bf T} {\bf C}_{\vp^{-1}}$ is a bounded operator. 

In terms of coefficients, it follows that
for fixed $k$, by \eqref{eq:smooth_v_k} and \eqref{eq:l_n_converges},
\begin{align*}
\widehat{v}_k = \lim_{m \to \infty}  \lim_{n \to \infty} \widehat{(v_{n,m})}_{k}  & =  2 \lim_{m \to \infty}  \lim_{n \to \infty} \sum_{|\ell| \le m} \widehat{\oo}_{k, \ell}(\vp_n) \, \widehat{\dot{u}}_{-\ell}\\
& = 2 \lim_{m \to \infty}  \sum_{|\ell| \le m} \widehat{\oo}_{k, \ell} (\varphi) \, \widehat{\dot{u}}_{-\ell}.
\end{align*}
Using $\widehat \l_{k, \ell} = \l_{k, \ell}/\sqrt{|k \ell|}$, we obtain the first equality in \eqref{eq:fourier_v}.
Re-indexing $\ell$ by $-\ell$, we obtain the second equality.  
\end{proof}

Recall that the coefficients $u_k$ with respect to the canonical basis in $\dot H^{1/2}$ are related to the Fourier coefficients by $u_k = \sqrt{|k|}\widehat u_k$ (see Section~\ref{sect: Fourier and H12}) and the Hilbert transform ${\bf J}$ is defined in \eqref{eq:Hilbert}. So we have
$$\frac{\widehat {\dot u}_\ell}{\sqrt{|\ell|}} = \ii \frac{\ell}{\sqrt{|\ell|}} \, \widehat u_\ell  = \ii \ \sgn (\ell) \sqrt{|\ell|}\, \widehat u_\ell = - ({\bf J} u )_\ell. $$ 
Hence, Lemma~\ref{lemma:Q-well-defined} can be restated as follows:  
\begin{lemma}\label{lem:Q-self-adjoint}
The matrix $-{\bf \OO}_\vp {\bf J}$ is the matrix representation of the bounded operator 
$$ \frac12 ({\bf T} -  {\bf C}_\vp {\bf T} {\bf C}_\vp^{-1})  \colon \dot H^{1/2} \to \dot H^{1/2}$$
with respect to the canonical basis, that is, identifying $\dot H^{1/2}$ with $\ell_{\mathbb{Z}^*}^2$ via coefficients.

The matrix $\OO_\vp$ defines a bounded operator  $\ell_{\mathbb{Z}^*}^2  \to \ell_{\mathbb{Z}^*}^2 $ by ``matrix multiplication'',
\[
 {\bf h}=\left((h_k)_{k \ge 1}; (h_{-k})_{k \ge 1} \right) \mapsto \OO_\vp{\bf  h} = \left((\OO_\vp{\bf h}_{k})_{k \ge 1}; (\OO_\vp{\bf  h}_{-k})_{k \ge 1}\right),
\]
where, for $|k| \ge 1$,
\[
\OO_\vp{\bf h}_{k} = \lim_{m \to \infty}\sum_{1 \le |\ell| \le m} \oo_{k,-\ell}h_\ell.
\]
Moreover, the operator $\OO_\vp$ is self-adjoint.

\end{lemma}
\begin{proof}
 The formula is a restatement of Lemma~\ref{lemma:Q-well-defined}.
    The self-adjointness was already observed in \eqref{eq:oo_self_adjoint}.
\end{proof}

We shall derive a version of the Grunsky inequalities for ${\bf \OO}$ from the following result. We could also use \eqref{eq:OO} below, but we find it interesting to use the single-layer potential directly, which makes the link to Dirichlet energies more apparent. Notice that the upper and lower bounds are not symmetric. See also Section~3 of \cite{CJV}.
\begin{proposition}
\label{lem:key-grunsky-estimate}
    Suppose $\vp \in \QS(\mathbb{S}^1)$ extends to a $K$-quasiconformal map of $\mathbb{D}$ and that $u \in \dot{H}^{1/2}$. Then
    \[
    \frac{1}{2\pi}\langle \dot{u}, {\bf C}_\vp {\bf T} {\bf C}_{\vp^{-1}} u \rangle_{L^2} = \|{\bf C}_{\vp^{-1}} u\|_{\dot H^{1/2}}^2.
    \]
    In particular, 
    \begin{align}  \label{eq:energy-identity}
    \frac{1}{2\pi}\langle \dot{u},({\bf T}- {\bf C}_\vp {\bf T} {\bf C}_{\vp^{-1}} ) u \rangle_{L^2} = \|u\|_{\dot H^{1/2}}^2-\|{\bf C}_{\vp^{-1}} u\|_{\dot H^{1/2}}^2
        \end{align}
   and consequently,
    \begin{align}\label{estimate:Grunsky-welding}
  -(K-1)\| u\|_{\dot H^{1/2}}^2 \le \frac{1}{2\pi}\langle \dot{u}, ({\bf T}-{\bf C}_\vp {\bf T} {\bf C}_{\vp^{-1}}) u \rangle_{L^2} \le ( 1-1/K)\| u\|_{\dot H^{1/2}}^2.
        \end{align}
\end{proposition}
\begin{proof}
 First, assume $\vp \in \Diff^\infty(\mathbb{S}^1)$ extends to a $K$-quasiconformal map of $\m D$ and that $u$ is smooth. Then, writing as before $\psi = \vp^{-1}$ and using 
\begin{align*}
   \partial_t (u\circ \psi (\ee^{\ii t}))  = \dot u \circ \psi(\ee^{\ii t})|\psi'(\ee^{\ii t})|,
\end{align*}
and a change of variables similar to \eqref{eq:CTC}, we obtain \begin{align*}
   \frac{1}{\pi} \langle \dot{u}, {\bf C}_\vp {\bf T} {\bf C}_{\vp^{-1}} u \rangle_{L^2} 
    & = \frac{1}{\pi^2}\iint_{[0,2\pi]^2}\partial_t{({\bf C}_{\psi} u)}(\ee^{\ii t}) \overline{\partial_s{({\bf C}_{\psi} u)}(\ee^{\ii s})} \log \left|e^{\ii s}-e^{\ii t}\right|^{-1} \dd s \dd t \\ 
    &  = 2\|{\bf C}_\psi u\|^2_{\dot H^{1/2}},
    \end{align*}
    where we used Lemma~\ref{lem:log-integral-dirichlet} for the last equality.
    The identity in the general case now follows by an approximation using Lemma~\ref{lem: qs approximation}.

    The estimates \eqref{estimate:Grunsky-welding} follow from \eqref{eq:energy-identity} and Lemma~\ref{lem:quasiinvariance}.
\end{proof}

\begin{corollary}
\label{cor:Grunsky-ineqs-general-version}
    Suppose $\vp \in \QS(\mathbb{S}^1)$ extends to a $K$-quasiconformal map of $\mathbb{D}$. If ${\bf h} = ((h_k)_{k \ge 1}; (h_{-k})_{k \ge 1}) \in \ell^2_{\mathbb{Z}^*}$, then,
    \[
    - \frac{1}{2}(K-1) \sum_{|k|\ge 1}|h_k|^2\le \lim_{m\to \infty} \sum_{1 \le |k|, |\ell| \le m} \oo_{k,- \ell}  \overline{h_{k}}h_\ell  \le \frac{1}{2}(1-1/K) \sum_{|k|\ge 1}|h_k|^2.
        \]
        In particular, since $\OO_\vp$ is self-adjoint,
\begin{equation*} 
\|\OO_\vp {\bf h}\|_{\ell^2} \le \frac{1}{2}\left(K-1\right)\|{\bf h}\|_{\ell^2}
\end{equation*}
and 
\begin{equation*}% \label{eq:norm_bound_OO}
{\bf I} - \OO_\vp \ge \frac12 \left(1 + \frac1K\right)
\end{equation*}
is invertible. 
\end{corollary}

\begin{proof}
Assume first $\vp \in \Diff^{\infty}(\mathbb{S}^1)$ and we prove the upper bound. 

Fix ${\bf h} = ((h_k)_{k \ge 1}; (h_{-k})_{k \ge 1}) \in \ell^2_{\mathbb{Z}^*}$,
for $|k| \ge 1$, we let $\widehat{u}_k \coloneqq \sqrt{|k|}h_k/(\ii k)$.  Then $u=\sum_{k \in \m Z^*} \widehat{u}_k \ee^{\ii k t} \in \dot H^{1/2}$ and we have 
\[\widehat{\dot{u}}_{k}/\sqrt{|k|}=h_k, \qquad \|u\|_{\dot H^{1/2}} = \|h\|_{\ell^2}.\] Let $u_m = P_mu$. 
By Lemma~\ref{lem:convergence of grunsky expansion} and \eqref{eq:CTC}, 
    \begin{align*}
    \frac{1}{4\pi}\langle \dot{u}_m, ({\bf T}-{\bf C}_\vp {\bf T} {\bf C}_{\vp^{-1}}) u_m  \rangle_{L^2} & = \frac{1}{(2\pi)^2}  \iint_{[0,2\pi]^2}   \dot{u}_m(\ee^{\ii s}) \overline{\dot{u}_m(\ee^{\ii t})}  \log \left|\frac{\varphi(\ee^{\ii s}) - \varphi(\ee^{\ii t})}{\ee^{\ii s} - \ee^{\ii t}} \right| \dd s \dd t\\
& =  \sum_{1 \le |k|, |\ell| \le m} \widehat{\oo}_{k,-\ell}  \overline{\widehat{\dot{u}}}_{k} \widehat{\dot{u}}_{\ell}\\
& =  \sum_{1 \le |k|, |\ell| \le m} \oo_{k,-\ell}  \frac{\overline{\widehat{\dot{u}}}_{k}}{\sqrt{|k|}} \frac{\widehat{\dot{u}}_{\ell}}{\sqrt{|\ell|}}\\
& = \sum_{1 \le |k|, |\ell| \le m} \oo_{k,-\ell}  \overline{h_k}h_\ell.
    \end{align*}
    On the other hand, by Lemma~\ref{lem:key-grunsky-estimate},
    \[
     \frac{1}{4\pi}\langle \dot{u}_m, ({\bf T}-{\bf C}_\vp {\bf T} {\bf C}_{\vp^{-1}}) u_m  \rangle_{L^2} = \frac{1}{2}\left(\|u_m\|_{\dot H^{1/2}}^2-\|{\bf C}_{\vp^{-1}} u_m\|_{\dot H^{1/2}}^2\right).
    \]
    By approximation, using Lemma~\ref{lem: qs approximation}, and then letting $m \to \infty$, we conclude that for $\vp \in \QS(\mathbb{S}^1)$,
    \begin{align}\label{march3}   
  \lim_{m\to \infty} \sum_{1 \le |k|, |\ell| \le m} \oo_{k,-\ell} \overline{h_{k}} h_\ell =  \frac{1}{2}(\|u\|_{\dot H^{1/2}}^2-\|{\bf C}_{\vp^{-1}} u\|_{\dot H^{1/2}}^2) \le \frac{1}{2}(1-1/K)\|u\|_{\dot H^{1/2}}^2.
        \end{align}
 The last estimate used the upper bound in \eqref{estimate:Grunsky-welding}. In particular, the limit on the left exists (by Lemma~\ref{lem:Q-self-adjoint}). We conclude using $\|u\|_{\dot H^{1/2}} = \|h\|_{\ell^2}$.

    The lower bound follows in a similar manner, using the lower bound in \eqref{estimate:Grunsky-welding}.
\end{proof}

\subsection{Proofs of characterizations and determinant formulas}\label{sect:proofs}

    Let $\varphi \in \QS(\mathbb{S}^1)$. We are now ready to prove Theorem~\ref{thm:2OO}. Recall that we want to show that  \begin{equation} \label{eq:OO}
         2\OO_\vp = {\bf I}-{\bf C}_{\vp} {\bf C}_{\vp}^*
    \end{equation}
                as matrices.
\begin{proof}[Proof of  Theorem~\ref{thm:2OO}]
    Since $-{\bf J}^2 = {\bf I}$, Lemma~\ref{lem:Q-self-adjoint} implies that $2\OO_\vp$ represents  in the canonical basis the operator:
    \[
  {\bf TJ} -  {\bf C}_\vp {\bf T} {\bf C}_\vp^{-1}{\bf J}:\dot{H}^{1/2} \to \dot{H}^{1/2} .
    \]
    Note first that ${\bf TJ}={\bf I}$: % in the canonical basis. 
     we have already observed that $\widehat{{\bf T}[u]}_k = \ii \, \sgn(k) \, \widehat{u}_k$. On the other hand,
     \[
        {\bf C}_{\vp^{-1}} = {\bf C}_{\vp}^{-1}= \begin{pmatrix} X^*&  -Y^t\\-Y^*&  X^t \end{pmatrix},
    \] 
    so
    \[
    {\bf T} {\bf C}_{\vp}^{-1} {\bf J} = - \begin{pmatrix} I &  0\\0&  -I \end{pmatrix} \begin{pmatrix} X^*&  -Y^t\\-Y^*&  X^t \end{pmatrix}\begin{pmatrix} -I &  0\\0&  I \end{pmatrix} =  \begin{pmatrix} X^*&  Y^t\\Y^*&  X^t \end{pmatrix}={\bf C}_\vp^*.
    \]
      This concludes the proof.
\end{proof}

       \begin{proof}[Proof of Theorem~\ref{thm:main-Q-with-Schatten}]
    Given \eqref{eq:OO}, we may compute as in  Lemma~2.7 of \cite{FS25_sle_welding}:
    \begin{align}\label{eq:expression of hat K}
2\OO_\vp & = {\bf I}-{\bf C}_{\vp} {\bf C}_{\vp}^*
    = {\bf I} -  \begin{pmatrix}
X & Y\\
\overline{Y} & \overline{X}
\end{pmatrix} \begin{pmatrix} X^*&  Y^t\\Y^*&  X^t \end{pmatrix} \nonumber \\
& = \begin{pmatrix} I - X X^* - Y Y^* &  - X Y^t - Y X^t\\ - \overline Y X^* - \overline X Y^* & I - \overline Y Y^t - \overline X X^t \end{pmatrix}  = 
\begin{pmatrix}  - 2 Y Y^* &  - 2 X Y^t \\ - 2 \overline X Y^* &  - 2 \overline Y Y^t \end{pmatrix} \nonumber \\
& =-2{\bf C}_{\vp} \begin{pmatrix} 0&  Y^t\\Y^* &  0 \end{pmatrix}
    \end{align}
    using the relations \eqref{XX* and I} and \eqref{eq:symplectic 2}. 

 We see immediately that we have the following equivalences: 
    $$\OO_\vp = 0 \Leftrightarrow {\bf C}_\varphi  \text{ is unitary}  \Leftrightarrow Y = 0 \Leftrightarrow \varphi \text{ preserves } H_+ \text { and } H_-  \Leftrightarrow  \varphi \in \mob(\m S^1). $$
    By Theorem~\ref{thm:comp_compact_HS}, we know that $Y$ is $p$-Schatten if and only if $\vp \in \WP_p(\mathbb{S}^1)$. Since ${\bf C}_{\vp}$ is bounded, this is also equivalent to $\OO_\vp$ being $p$-Schatten by \eqref{eq:expression of hat K}. Similarly, $\OO_\vp$ is compact if and only if $Y$ is compact, which, again by  Theorem~\ref{thm:comp_compact_HS}, holds if and only if $\vp \in \Sym(\mathbb{S}^1)$. 
       \end{proof}

Another immediate corollary of \eqref{eq:OO} is the following composition rule for $\OO_\vp$.
\begin{corollary}\label{cor:OO_two_vp}
    If $\vp_1, \vp_2 \in \QS$, we have $\OO_{\varphi_1 \circ \varphi_2} = {\bf C}_{\vp_2}{\OO}_{\vp_1}{\bf C}_{\vp_2}^* + {\OO}_{\vp_2}$. In particular, if $\mu \in \mob(\m S^1)$, $\OO_{\varphi\circ \mu}=  {\bf C}_\mu \OO_\varphi  {\bf C}_\mu^*$ is a conjugation of $\OO_\vp$ by the unitary operator ${\bf C}_\mu$. Moreover, $\OO_{\vp_1} = \OO_{\vp_2}$ if and only if there exists $\mu \in \mob(\m S^1)$ such that $\varphi_1 = \mu \circ \varphi_2$.
\end{corollary}
\begin{proof}
    By \eqref{eq:OO}, we have
    \begin{align*}
        2 \OO_{\varphi_1 \circ \varphi_2} & = {\bf I} - {\bf C}_{\vp_2} {\bf C}_{\vp_1} {\bf C}_{\vp_1}^* {\bf C}_{\vp_2}^* = {\bf I} -  {\bf C}_{\vp_2} {\bf C}_{\vp_2}^* +  {\bf C}_{\vp_2} ({\bf I} -  {\bf C}_{\vp_1} {\bf C}_{\vp_1}^* ){\bf C}_{\vp_2}^* 
        \\ & = 2 {\OO}_{\vp_2} + 2 {\bf C}_{\vp_2}{\OO}_{\vp_1}{\bf C}_{\vp_2}^*,
    \end{align*}
    which proves the identity.

    We have seen that if $\mu \in \mob (\m S^1)$, then ${\bf C_\mu}$ is unitary and $\OO_\mu = 0$ by Lemma~\ref{lem:mob}. This gives $\OO_{\varphi\circ \mu}=  {\bf C}_\mu \OO_\varphi  {\bf C}_\mu^*$.

    Now we assume $\OO_{\vp_1} = \OO_{\vp_2}$, then 
    $$\OO_{\vp_1} = \OO_{(\vp_1 \circ \vp_2^{-1}) \circ \vp_2} = \OO_{\vp_2} + {\bf C}_{\vp_2} \OO_{\vp_1 \circ \vp_2^{-1}}  {\bf C}_{\vp_2}^*$$
    implies that $\OO_{\vp_1 \circ \vp_2^{-1}}  = 0$ since ${\bf C}_{\vp_2}$ is invertible. Theorem~\ref{thm:main-Q-with-Schatten} then implies $\vp_1 \circ \vp_2^{-1} \in \mob (\m S^1)$ and completes the proof.
\end{proof}

          We are now ready to prove Theorem~\ref{thm:main-det2}. 
\begin{proof}[Proof of Theorem~\ref{thm:main-det2}]
By \eqref{eq:OO} and Lemma~\ref{lem:singular value decomposition}, there exists unitary matrix $U$, acting on $\ell^2_{\m Z_+}$, and a diagonal matrix $D$ with $0 \le D \le \delta(\varphi)<1$, such that 
\begin{align*}
    2 \OO_\vp = {\bf I} - {\bf U} {\bf D} {\bf U}^*
\end{align*}
where 
$${\bf U} = \frac{1}{\sqrt 2}  \begin{pmatrix} U &  U\\ \overline U & -  \overline{U}  \end{pmatrix} \quad \text{ is unitary on } \ell^2_{\m Z^*} \text { and } \quad {\bf D} =  \begin{pmatrix} \frac{I + D}{I - D} &  0\\ 0 & \frac{I - D}{I + D} \end{pmatrix}.$$
In particular, this shows that
\begin{align}\label{unitary}
   {\bf I} -  \OO_\vp =\frac12 {\bf U} ({\bf I + D} ){\bf U}^*
   %= {\bf U} \begin{pmatrix} I -\frac{-D}{I - D}  &  0\\ 0 & I - \frac{D}{I + D} \end{pmatrix} {\bf U}^* 
   = {\bf U} \begin{pmatrix} \frac{I}{I - D}  &  0\\ 0 & \frac{I}{I + D}  \end{pmatrix} {\bf U}^* = {\bf U} \begin{pmatrix} I - D  &  0\\ 0 & I + D  \end{pmatrix}^{-1}{\bf U}^*.
\end{align}
The fact that the entries $(\delta_k)_k$ of $D$ equal the reciprocal positive Fredholm eigenvalues of the welded curve $\gamma$ was explained in the remark following Lemma~\ref{lem:singular_value}. On the other hand, by \eqref{unitary} and \eqref{eq:norm_bound_OO} we have
\[
\frac{1}{1+\delta_k} \ge \frac{1}{2}\left(1 + \frac{1}{K}\right),
\]
which implies $\delta_k \le (K-1)/(K+1)$.

It remains to prove the identity involving the determinants when $\vp \in \WP(\mathbb{S}^1)$.

For this, we know from Theorem~\ref{thm:main-Q-with-Schatten} that $\OO_\vp$ is Hilbert--Schmidt and by Proposition~\ref{prop:qs_kappa_bound}, ${\bf I} - \OO_\vp$ is invertible. We have that $({\bf I} -\OO_\vp)^{-1} = {\bf I} + \OO_\vp({\bf I} - \OO_\vp)^{-1}$. Hence the regularized Fredholm determinant $\det\nolimits_2({\bf I} -\OO_\vp)^{-1}$ is well-defined. Let $(\eta_k)_{k \ge 1}$ be the eigenvalues of $\OO_\vp({\bf I} - \OO_\vp)^{-1}$. Then, by definition and using \eqref{unitary}, 
\[
\det\nolimits_2 \left({\bf I} -\OO_\vp \right)^{-1} = \prod_{k=1}^\infty(1+\eta_k)\ee^{-\eta_k} = \prod_{k=1}^\infty(1-\delta_k^2) = \det (I-D^2),
\]
where we also used that the exponential factors cancel for the second identity, since $(\eta_k)_{k \ge 1} = (\pm \delta_k)_{k \ge 1}$. By Corollary~\ref{cor:energy_D},
$- \log \det(I-D^2) = I^L(\gamma)/12$, which concludes the proof.
\end{proof}

We finally come to the proof of Proposition~\ref{prop:main-Q2}, which expresses $I^L$ using the top left block of $\OO_\vp$.
   \begin{proof}[Proof of Proposition~\ref{prop:main-Q2}]
            By \eqref{eq:expression of hat K}, we have $M=-YY^*$.
            Since $Y$ is Hilbert--Schmidt if and only if $\vp \in \WP(\mathbb{S}^1)$ (Theorem~\ref{thm:comp_compact_HS}), it follows that $M$ is of trace class if and only if $\vp \in \WP(\mathbb{S}^1)$. 
             We obtain from Proposition~\ref{prop:IL_X_Y} that 
\[
12\log \det (I-M) = 12\log \det (I + YY^*) = I^L(\gamma),
\] 
which completes the proof.\end{proof}

\section{Dirichlet energy formula}\label{sect:dirichlet}

\subsection{Proof of Theorem~\ref{thm:DE}}
Let $\vp \in \QS(\mathbb{S}^1)$. One approach to the welding problem (see \cite{david_chord_arc}) derives equations for $(f,g)$ in terms of $\vp$, starting from the desired equation ${\bf U}_\vp g = f$. Here we will use similar ideas in order to express the $\dot{H}^{1/2}$-norms of $\log|f'|$ and $\log|g'|$ in terms of quantities directly involving $\log \vp'$. We will write ${\bf R_0} u := u-\widehat{u}_0$ for the operator that removes the zero-mode. With this notation, ${\bf C}_\vp = {\bf R}_0 {\bf U}_\vp$. 
%We begin with a lemma.  
\begin{lemma}\label{lemma:dirichlet-translation}
Let $\vp \in \WP(\mathbb{S}^1)$ and write $\psi = \vp^{-1}$. Let $(f,g)$ be the solution to the welding problem normalized so that $f(0)=0, f'(0)=1$ and $g(\infty) = \infty$.
 Let $B_2, B_3$ be the classical Grunsky matrices associated with $(f,g)$.
 Then,
    \[
    (\log f')_+ = -(X^*)^{-1}(\log \psi')_+, \qquad (\log g')_- = - (\overline{X})^{-1} (\log \vp')_-
    \]
    In particular,
\begin{align}\label{eq: log f' log g'}
    \|\log |f'|\|_{\dot H^{1/2}}^2 = \frac{1}{2}\| (X^*)^{-1} (\log \psi')_+\|_{H_+}^2, \quad \|\log |g'|\|^2_{\dot H^{1/2}}  = \frac{1}{2}\| (\overline{X})^{-1} (\log \vp')_-\|_{H_-}^2.
    \end{align}
\end{lemma}

\begin{proof}
We have ${\bf U}_\vp g =f$, so by the chain rule
\[
{\bf U}_\vp \log g' + \log \vp' = \log f'.
\]
Hence,
\begin{align*}
{\bf C}_\vp \log g' + {\bf R_0}\log \vp' & = {\bf R_0}\log f'.
\end{align*}
Since $\log f'$ is analytic in $\mathbb{D}$, it follows that
\begin{align*}
{\bf P_+}({\bf C}_\vp \log g') + {\bf P_+}\log \vp' & = {\bf P_+}\log f' \\
{\bf P_-}({\bf C}_\vp \log g') + {\bf P_-}\log \vp' & = 0.
\end{align*}
 We can write the second equation as
 \[
 \overline{X} {\bf P_-}\log  g'  = -  {\bf P_-}\log \vp' \Leftrightarrow {\bf P_-}\log  g' = -\overline{X}^{-1}{\bf P_-}\log \vp'.
 \]
  Note that ${\bf P_-} \log g' = {\bf R_0} \log g'$. Therefore, 
\[
{\bf R_0} \log g' = (\log g')_- = - \overline X^{-1}(\log \vp')_- 
\]
We get the second identity of \eqref{eq: log f' log g'} by noting that $\| \log|g'| \|_{\dot H^{1/2}}^2 = \| \log g' \|_{\dot H^{1/2}}^2/2$. 

The identity for $\log|f'|$ is proved similarly starting from ${\bf U}_\psi f = g$, and we obtain 
$$(\log f')_+ = - (X^*)^{-1} (\log \psi')_+,$$
as claimed.
\end{proof}

\begin{rem}
One can also start the proof of Lemma~\ref{lemma:dirichlet-translation} from the formula ${\bf U}_\vp g =f$, and differentiate, but then split into real and imaginary parts:
\[
{\bf C}_\vp \log |g'| +  {\bf R}_0\log |\vp'| = {\bf R}_0\log| f'|
\]
and
\[
{\bf C}_\vp \, \im \log  g' +  {\bf R}_0 \, \im \log \vp' = {\bf R}_0 \, \im \log  f'.
\]
Then use the relation between the conjugate function and the Hilbert transform ${\bf J}$. This leads to an interesting equation involving the anticommutator $\{{\bf J}, {\bf C}_\vp\}$,
     \begin{align}\label{eq:GD-main-eq}
 ({\bf J}{\bf C}_\vp + {\bf C}_\vp {\bf J}){\bf R}_0 \log |g'| = {\bf R}_0 \, \im \log \vp' - {\bf J}{\bf R}_0\log |\vp'|.
\end{align}
This is an analog of an equation obtained by G. David in \cite{david_chord_arc}. In the present setting, we can use the block matrix representations for ${\bf C}_\vp$ and ${\bf J}$ to see that
 \[
  {\bf J}{\bf C}_\vp + {\bf C}_\vp {\bf J}= \ii \left( \begin{pmatrix}
-X & Y\\
-\overline{Y} & \overline{X}
\end{pmatrix} +  \begin{pmatrix}
-X & -Y\\
\overline{Y} & \overline{X}
\end{pmatrix}\right) = 2\ii \begin{pmatrix}
-X & 0\\
0 & \overline{X}
\end{pmatrix}.
 \]
 By Lemma~\ref{lem:composition-grunsky-closed}, $X$ is invertible, and $X^{-1}=\overline{B}_3$ so 
 \[
 ({\bf J}{\bf C}_\vp + {\bf C}_\vp {\bf J})^{-1}= \frac{1}{2\ii} \begin{pmatrix}
-X^{-1} & 0\\
0 & \overline{X}^{-1}
\end{pmatrix}
 \]
 and equation~\eqref{eq:GD-main-eq} can thus be solved for ${\bf R}_0 \log|g'|$ in terms of quantities involving only the real and imaginary parts of $\log \vp'$. After some manipulations, this leads to the identity for $\|\log|g'|\|^2_{\dot H^{1/2}}$ in \eqref{eq: log f' log g'}. The other term in \eqref{eq: log f' log g'} is obtained in a similar way.
\end{rem}

The next lemma uses an idea from \cite{SchippersStaubach2015} to show that, just like the universal Liouville action/Loewner energy, the Velling--Kirillov potential $\log \left| f'(0)/g'(\infty) \right|$, see \cite{TT06}, can also be expressed as a determinant (albeit of a minor with one element) involving a Grunsky matrix. 

\begin{lemma}\label{lemma:VK}
   Consider the set-up of Lemma~\ref{lemma:dirichlet-translation}. Then, \[f'(0)/g'(\infty) = (B_2)_{1,1} =(B_3)_{1,1}.\] In particular, 
    $
\log \left| f'(0)/g'(\infty)\right| 
\le 0.$
\end{lemma}
\begin{proof}
 We can write $f(z) = %\alpha 
 z + O(z^2)$ and $g(z) = \beta z + g_0 + g_-(z)$. Here $g_-$ is holomorphic in $\mathbb{D}^*$ and $g_-(z) \to 0$ as $z \to \infty$. Furthermore, we have ${\bf R}_0 g = \beta z + g_-$ and ${\bf P}_+ f = f$. Note also that $\log|g'(\infty)| = \log|({\bf R}_0 g)'(\infty)|= \log \beta$. 
 
 Recall that $\vp = g^{-1} \circ f \mid_{\mathbb{S}^1}$ and $\psi = \vp^{-1} = f^{-1}\circ g\mid_{\mathbb{S}^1}$. By definition, ${\bf U}_\psi f = g$. Since ${\bf C}_\psi = {\bf R}_0{\bf U}_\psi$ and ${\bf P}_+ f = f$, we therefore see that 
 \begin{align}\label{eq:f-g}
 {\bf C}_\psi {\bf P}_+f = {\bf C}_\psi f = {\bf R}_0{\bf U}_\psi f = {\bf R}_0 g = \beta z + g_-.
 \end{align}  
 Write $X_\psi, Y_\psi$ for the components of ${\bf C}_\psi$ as in \eqref{eq:composition-matrix1}. We represent $f$ and ${\bf R}_0 g$ as vectors,
 \[
 {\bf f} = (f_+, 0), \quad {\bf R}_0{\bf g} = (g_+, g_-),
 \]
 respectively, where,
 \[
 f_+=(1, f_2, \ldots), \qquad g_+=(\beta, 0, \ldots).
 \]
 With this notation, we deduce from \eqref{eq:f-g} that
\[
X_\psi f_+   = g_+  = \b z \Longrightarrow  1 =  (X_\psi^{-1})_{1,1} \beta.
\]
Here we used the fact that $X_\psi$ is invertible, which follows from Lemma~\ref{lem:composition-grunsky-closed}.
On the other hand, we know from \eqref{eq:composition-matrix1} that $X_\psi = X_{\vp}^*$. So using Theorem~\ref{thm:relation between Grunsky and composition}, we see $X_\psi^{-1} = (X_{\vp}^*)^{-1} = B_2$. It follows that $$f'(0)/g'(\infty) = 1/\b = (B_2)_{1,1} =(B_3)_{1,1},$$
where the last identity uses $B_2=B_3^t$. The fact that $\log| 1/\beta| \le 0$ finally follows from the last equation and that $B_2$ is a contraction. 
\end{proof}
\begin{rem}
    Recall the definition of classical Grunsky coefficients $a_{k,\ell}$ in Section~\ref{sect:classical_grunsky}, $f'(0)/g'(\infty)$ can be expressed as $\exp(-a_{0,0})$. However, $a_{0,0}$ is not one of the entries of the classical Grunsky matrix, and the relation to the composition operator is not clear from this. 
    Another way to obtain Lemma~\ref{lemma:VK} is to notice that
    \begin{align*}
        a_{1,-1} & = -\lim_{z \to \infty} z \lim_{w \to 0} \frac{1}{w}\left(\log \frac{g(z) - f(w)}{z} - \log \frac{g(z) - f(0)}{z}\right)  \\
        & = -\lim_{z \to \infty} \frac{- z f'(0)}{g(z) - f(0)} =  \lim_{z \to \infty} \frac{ z}{g(z)} = \frac{1}{g'(\infty)}.  
    \end{align*}
 However, we find the longer proof interesting, as it directly uses the composition operator. 
\end{rem}

\begin{proof}[Proof of Theorem~\ref{thm:DE}]
The formula follows immediately from the proof of Lemma~\ref{lemma:dirichlet-translation} and Lemma~\ref{lemma:VK}. 
\end{proof}
\subsection{Proof of Corollary~\ref{cor:growth}}
We are now ready to prove the corollary to Theorem~\ref{thm:DE}.
\begin{proof}[Proof of Corollary~\ref{cor:growth}]
    By Theorem~\ref{thm:DE}, and using that the Velling--Kirillov potential is non-positive (Lemma~\ref{lemma:VK}), we have
    \begin{align} \label{eq:energy_comp}
    I^L(\vp^{\circ n}) \le \|B_2[\psi^{\circ n}] (\log (\psi^{\circ n})')_+\|_{H_+}^2 + \|B_3[\vp^{\circ n}] (\log (\vp^{\circ n})')_-\|_{H_-}^2,
        \end{align}
    where we write $B_2[\psi^{\circ n}]$ for the classical Grunsky matrix (acting on $H_+$) corresponding to the normalized pair $(f,g)$ obtained from $\psi^{\circ n}$ by conformal welding and similarly for $B_3$.
    We have
    \[
    \log (\vp^{\circ n})'(z) = \sum_{j=0}^{n-1}\log \vp'(\vp^{\circ j}(z)).
    \]
    Hence, using Lemma~\ref{lem:quasiinvariance},
    \[
    \|\log (\vp^{\circ n})'\|_{\dot H^{1/2}} \le \sum_{j=0}^{n-1}\|({\bf C}_\vp)^j  \log \vp'\| \le \|\log \vp'\|_{\dot H^{1/2}}  \sum_{j=0}^{n-1}K^{j/2} = \| \log \vp'\|_{\dot H^{1/2}}  \frac{K^{n/2}-1}{K^{1/2}-1}.
    \]
    By Lemma~\ref{lem:Grunsky_eq}, viewing $B_3$ as an operator acting on $H^{1/2}$, and using that ${\bf P_-}$ is a projection, we have the estimate
    \[\|B_3[\vp^{\circ n}] (\log (\vp^{\circ n}))'_-\|_{H_-} \le \|\log (\vp^{\circ n})'\|_{\dot H^{1/2}}. \]
    Hence,
    \[
    \|B_3[\vp^{\circ n}] (\log (\vp^{\circ n}))'_-\|_{H_-} \le \|\log (\vp^{\circ n})'\|_{\dot H^{1/2}} \le  \frac{K^{n/2}-1}{K^{1/2}-1} \| \log \vp'\|_{\dot H^{1/2}}.
    \]
    We argue similarly for $\log (\psi^{\circ n})'$ and obtain the analogous bound. By \eqref{eq:energy_comp} we get
    \[
    I^L(\vp^{\circ n}) \le \left(\frac{K^{n/2}-1}{K^{1/2}-1}\right)^2\left(\| \log \psi'\|_{\dot H^{1/2}}^2 +  \| \log \vp'\|_{\dot H^{1/2}}^2 \right),
    \]
    as claimed.
\end{proof}

\bibliographystyle{plain}%alpha
\bibliography{ref}

\end{document}